\documentclass[11pt,a4paper]{amsart}
\usepackage{amssymb,latexsym}
\usepackage{amsfonts}
\usepackage{amsmath}
\usepackage{enumitem}
\usepackage{microtype}

\usepackage{tikz-cd}
\usepackage{amsthm}
\usepackage[square, comma, sort&compress, numbers]{natbib} %管理文献
\usepackage[marginal]{footmisc}
\usepackage{mathrsfs} %花体
\usepackage[colorlinks=true,linkcolor=blue,citecolor=blue,urlcolor=blue]{hyperref}
\newtheorem{theorem}{Theorem}[section]
\newtheorem{corollary}[theorem]{Corollary}
\newtheorem{lemma}[theorem]{Lemma}
\newtheorem{proposition}[theorem]{Proposition}
\newtheorem{definition}[theorem]{Definition}

\newtheorem{notation}[theorem]{Notation}

\newtheorem{remark}[theorem]{\bf{Remark}}
\newtheorem{claim}[theorem]{Claim}

\newtheorem{example}[theorem]{Example}
\numberwithin{equation}{section}

\newcommand{\bp}{\bar{\partial}}
\newcommand{\beq}{\begin{equation}}
	\newcommand{\eeq}{\end{equation}}
\newcommand{\beqn}{\begin{equation*}}
	\newcommand{\eeqn}{\end{equation*}}
\newcommand{\vps}{\varepsilon}

\newcommand{\C}{\mathbb{C}}
\newcommand{\R}{\mathbb{R}}
\newcommand{\N}{\mathbb{N}}

\newcommand{\reg}{\mathrm{reg}}
\newcommand{\torsion}{\mathrm{torsion}}
\newcommand{\ch}{\mathrm{ch}}
\newcommand{\codim}{\mathrm{codim}}
\newcommand{\orb}{\mathrm{orb}}
\newcommand{\pp}{\partial\bar{\partial}}
\newcommand{\p}{\partial}

\newcommand{\w}{\omega}

\newcommand{\im}{\sqrt{-1}}

\newcommand{\mF}{\mathcal{F}}
\newcommand{\mO}{\mathcal{O}}
\newcommand{\mT}{\mathcal{T}}
\newcommand{\mE}{\mathcal{E}}
\newcommand{\mL}{\mathcal{L}}
\newcommand{\mC}{\mathcal{C}}
\newcommand{\mQ}{\mathcal{Q}}
\newcommand{\mG}{\mathcal{G}}
\newcommand{\mS}{\mathcal{S}}

\newcommand{\image}{\operatorname{im}}

\DeclareMathOperator \Vol{Vol}

\DeclareMathOperator \End{End}
\DeclareMathOperator \Id{Id}

\DeclareMathOperator \rank{rank}
\DeclareMathOperator \tr{tr}

\begin{document}
	
	\vspace*{0.5cm} \begin{center}\noindent {\LARGE \bf The Miyaoka-Yau inequality for minimal K\"ahler klt spaces}\\[0.7cm]
			
	Chuanjing Zhang\footnotemark[1]%\footnote{\begin{tabular}{@{}r@{}p{13.4cm}@{}}&\end{tabular}}
	\\[0.2cm]
	School of Mathematics and Statistics\\
	Ningbo University\\ Ningbo, 315211, P.R. China\\
	E-mail: zhangchuanjing@nbu.edu.cn\\[0.5cm]
	Shiyu Zhang\footnotemark[1]
	%\footnote{\begin{tabular}{@{}r@{}p{13.4cm}@{}} &The research was supported by the National Key R and D Program of China 2020YFA0713100, All authors were supported in part by NSF in China, No.11625106, 12001548, 11571332 and 11721101. The first proof of this paper was completed when the first author was at School of Mathematical Sciences, University of Science and Technology of China, and the second author was at  School of Mathematical Sciences, Xiamen University.\end{tabular}}
	\\[0.2cm]
	School of Science, Institute for Theoretical Sciences\\
	Westlake University\\ Hangzhou
310030, China\\
	E-mail: shiyu123@mail.ustc.edu.cn\\[0.5cm]
	Xi Zhang\footnotemark[1]%\footnote{\begin{tabular}{@{}r@{}p{13.4cm}@{}}&\end{tabular}}
	\\[0.2cm]
    School of Mathematics and Statistics\\
	Nanjing University of Science and Technology\\
	Nanjing, 210094, P.R.China\\
	E-mail: mathzx@njust.edu.cn\\[1cm]
	\footnotetext[1]{\notag\noindent The research was supported by the National Key R and D Program of China 2020YFA0713100. The  authors are partially supported by NSF in China No.12141104, 12371062 and 12431004.}
			
	\end{center}
%%%%%%%%%%%%%%%%%%%%%%%%%%%%%%%

	{\bf Abstract.} Let $X$ be a compact K\"ahler space with quotient singularities in codimension $2$. In terms of orbifold Chern classes, we prove a generalized Bogomolov-Gieseker inequality for a reflexive coherent sheaf $\mathcal{E}_X$ on $X$ equipped with a Higgs field $\theta_{X_\reg}$ defined on the regular locus of $X$. As a direct application, we prove the Miyaoka-Yau inequality for minimal K\"ahler spaces.
	
	The proof relies on two main ingredients. 
	The first is the existence of $L^p$-approximate critical Hermitian structures for orbifold Higgs bundles over compact Gauduchon orbifolds.
	The second is the construction of an orbifold Higgs bundle on a suitable orbifold modification \(f:W\to X\), agreeing with the reflexive pullback of $(\mE_X,\theta_{X_\reg})$ over \(f^{-1}(X_\reg)\).
	\\[1cm]

\newpage
\tableofcontents

	%%%%%%%%%%%%%%%%%%%%%%%%%%%%%%%

\section{Introduction}

\subsection{Miyaoka-Yau inequality for minimal K\"ahler klt spaces}
Yau's solution of Calabi's conjecture (\cite{Yau78}) implies the following Chern number inequality for a compact K\"ahler manifold $X$ of complex dimension $n$ with $K_X$ ample:
\begin{equation}\label{smoothmiyaokayau}
	(2c_2(X)-\frac{n}{n+1}c_1(X)^2)\cdot K_X^{n-2}\geq0,
\end{equation}
which was also proved by Miyaoka (\cite{miyaoka77}) for complex surfaces of general type. The inequality \eqref{smoothmiyaokayau} is called the Miyaoka--Yau inequality. Using $\mathbb{Q}$-Chern classes, Greb--Kebekus--Peternell--Taji and Guenancia--Taji (\cite{GKPT19a,GT22}) established the Miyaoka--Yau inequality for all minimal projective klt varieties.

\vspace{0,1cm}

It is natural to ask whether the Miyaoka-Yau inequality holds for all minimal K\"ahler spaces, without assuming projectivity. The methods of \cite{GKPT19a,GT22} do not apply directly, since hypersurface sections are not available in this setting. For a nef class $\eta\in H_{BC}^{1,1}(X)$ on a compact K\"ahler space $(X,\w_X)$ of dimension $n$, the numerical dimension $v$ of $\eta$ is defined by
$$v(\eta):=\max\{k=0,1,\cdots,n: \eta^k\cdot[\w_X]^{n-k}>0\},$$
which is independent of the choice of $\w_X$. In this paper, we prove the following statement.

\begin{theorem}[Orbifold Miyaoka-Yau inequality]\label{t1}
	Let $X$ be a compact K\"ahler space of dimension $n$ with klt singularities and nef canonical sheaf. Then for any K\"ahler form $\w_X$ on $X$, we have
	\begin{equation}\label{miyaokayauformula}
		\begin{split}
			(2\widehat{c}_2(X)-\frac{n}{n+1}\widehat{c}_1(X)^2)\cdot (K_X)^{i}\cdot[\w_X]^{n-2-i}\geq0,
		\end{split}
	\end{equation}
	where $i=\min(v(K_X),n-2)$, $\widehat{c}_2(X)=\widehat{c}_2(\mT_X)$ and $\widehat{c}_1(X)^2=\widehat{c}_1(\mT_X)^2$ denote the orbifold Chern classes of $X$.
\end{theorem}

In the smooth case, this result was obtained in \cite{Liu2023} using cscK metrics. For singular spaces with $K_X$ numerically trivial, it was obtained in \cite[Theorem 1.6]{CGG24}, using the decomposition theorem of \cite{BGL22}. See \cite{Tsuji88,zhangyuguang2009,Song2016} and the references therein for earlier generalizations in the smooth setting. When $K_X$ is not big, \eqref{miyaokayauformula} takes the form
\begin{equation}\label{equa-miyaoka}
 \widehat c_2(X)\cdot(K_X)^i\cdot[\w_X]^{n-2-i}\geq0.
\end{equation}
Such second Chern class inequalities play an important role in the proof of abundance for threefolds; see, for instance, \cite[Theorem 1.1]{Miyaoka88} and \cite[Theorem 0.3]{DP03} for the projective and K\"ahler settings, respectively. More generally, one expects the second Chern class to be nonnegative with respect to arbitrary products of nef classes. In the projective setting, this was proved in \cite{Miyaoka87} and extended to klt varieties in \cite{IMM25}. The central issue is to establish generic semipositivity of $\Omega_X^{[1]}$ with respect to arbitrary nef classes $\alpha_0,\ldots,\alpha_{n-2}$, which remains unknown for K\"ahler klt spaces.

\subsection{Bogomolov-Gieseker inequality for Higgs sheaves over singular K\"ahler spaces}
The proof of Theorem~\ref{t1} relies on a Bogomolov--Gieseker (BG) inequality for Higgs sheaves on singular K\"ahler spaces. 

We recall the background to the BG inequality. 
Let $\mE$ be a reflexive coherent sheaf of rank $r$ on a projective manifold $X$ of dimension $n$. Set
\[
 \Delta(\mE):=2c_2(\mE)-\frac{r-1}{r}c_1(\mE)^2.
\]
A well-known form of Bogomolov's inequality states that if $\mE$ is semistable with respect to an ample divisor $H$, then
\begin{equation}\label{equa-Bo}
 \Delta(\mE)\cdot H^{n-2}\geq0.
\end{equation}
This was first proved by F. Bogomolov in the surface case; the higher-dimensional case follows from the Mehta--Ramanathan restriction theorem. For a compact K\"ahler manifold $(X,\w)$, the corresponding inequality with $H$ replaced by $[\w]$ follows for $[\w]$-stable vector bundles from the Donaldson--Uhlenbeck--Yau theorem (\cite{DON85,UY86}), which asserts the existence of a Hermitian--Einstein metric. Bando--Siu (\cite{BS}) generalized this to stable reflexive sheaves via admissible Hermitian--Einstein metrics. Simpson (\cite{Simpson88}) extended the Donaldson--Uhlenbeck--Yau theorem to stable Higgs bundles and obtained the corresponding BG inequality, which provided an alternative approach to the Miyaoka--Yau inequality \eqref{smoothmiyaokayau}. Later, BG inequalities were obtained for semistable Higgs bundles and semistable reflexive Higgs sheaves; cf. \cite{LZ,LZZ}.

The BG inequality for semistable reflexive Higgs sheaves was extended to projective klt varieties in \cite{GKPT19a}, where the Miyaoka--Yau inequality was also obtained as an application by developing Simpson's idea. In terms of orbifold Chern classes, the BG inequality for stable reflexive sheaves was established on compact K\"ahler klt spaces of dimension $3$ in \cite{GP24} and on compact K\"ahler spaces with quotient singularities in codimension $2$ in \cite[Theorem 1.2]{Ou25}. These developments are partly motivated by abundance for K\"ahler threefolds (cf. \cite[Appendix]{GP24}) and quasi-uniformization of klt spaces (cf. \cite{GK20,CGG24}). We also refer to \cite{Wu21,CGG22,CGNPPW23,Chen25,Jin25} and \cite{Lan24} for other generalizations on normal K\"ahler spaces and normal projective varieties in positive characteristic, respectively.

\vspace{0.1cm}

In this paper, we extend the BG inequality to reflexive Higgs sheaves over compact K\"ahler spaces with quotient singularities in codimension $2$.

\begin{theorem}[Orbifold Bogomolov--Gieseker inequalities]\label{singulargeneralized}
	Let $X$ be a compact normal K\"ahler space of dimension $n\geq2$ with quotient singularities in codimension $2$. Let $\mE_X$ be a reflexive coherent sheaf of rank $r$ on $X$, equipped with a Higgs field $\theta_{X_\reg}$ over the regular locus of $X$. Let $
	  \widehat\Delta(\mE_X):=
	  2\widehat c_2(\mE_X)-\frac{r-1}{r}\widehat c_1(\mE_X)^2
	$. The following statements hold.
	\begin{enumerate}[label=\textup{(\roman*)}]
	\item If $\alpha$ is nef and big, then
	\begin{equation}\label{Bo3}
	  \widehat\Delta(\mE_X)\cdot\alpha^{n-2}
	  \geq -\frac{1}{\alpha^n}
	  \sum_{i=1}^r\bigl(\mu_{i,\alpha}-\mu_\alpha(\mE_X)\bigr)^2,
	\end{equation}
	 where $\mu_{i,\alpha}\ (1\leq i\leq r)$ are the Harder-Narasimhan slopes of $(\mE_X,\theta_{X_\reg})$ with respect to $\alpha$.
	\item If
	$(\mE_X,\theta_{X_\reg})$ is semistable with respect to some K\"ahler forms $\omega_0,\ldots,\omega_{n-2}$, then
	\begin{equation}\label{Bo3-mixed}
	  \widehat\Delta(\mE_X)\cdot[\omega_1]\cdots[\omega_{n-2}]\geq0.
	\end{equation}
	\end{enumerate}
\end{theorem}

Our proof of the BG inequality starts from a partial orbifold resolution $f:Y\rightarrow X$ constructed in \cite[Theorem 1]{KO25}: a projective bimeromorphic morphism from a compact K\"ahler space $Y$ with only quotient singularities such that the indeterminacy locus of $f^{-1}$ has codimension at least $3$.

A new ingredient is the pull-back of the Higgs field to a suitable orbifold modification. More precisely, we construct a compact K\"ahler orbifold $W_\orb$, a projective bimeromorphic morphism $h:W\rightarrow X$ from its coarse space, and a Higgs orbi-bundle $(\mE_\orb,\theta_\orb)$ on $W_\orb$ such that $h$ is an isomorphism over $U$ and $(E_\orb,\theta_\orb)|_{h^{-1}(U)}$ is the reflexive orbifold lift of $(E_X,\theta_{X_\reg})$. The Higgs field on this reflexive orbifold lift, induced by $\theta_{X_\reg}$, extends holomorphically across $h^{-1}(Z)$ to the globally defined field $\theta_\orb$. The key step is to apply the extension property for reflexive differential $1$-forms obtained in \cite{Heuer25} over spectral covers; see Proposition~\ref{prop-higgs-orbifold-model}.

As in \cite{Ou25}, the problem then reduces to proving the quantitative BG inequality for the Higgs orbi-bundle $(E_\orb,\theta_\orb)$ (see Corollary~\ref{coro-astheno}). To this end, we construct $L^p$-approximate critical Hermitian structures on Higgs orbi-bundles in Section~\ref{higgsorbi-bundle}, extending the argument of \cite{LZZ21}. This gives an analytic proof of the generalized BG inequality.

\begin{remark}[Motivation for restricting Higgs fields to the regular locus]
	In their work on projective varieties with klt singularities, Greb-Kebekus-Peternell-Taji \cite{GKPT19a} define a Higgs field on a coherent sheaf $\mE_X$ as a morphism
    $\theta_X:\mE_X\rightarrow \mE_X\otimes\Omega_X^{[1]}$
    such that the composition
    $$\theta\wedge\theta:\mE_X\xrightarrow{\theta}\mE_X\otimes \Omega_X^{[1]}\xrightarrow{\theta\otimes\Id} \mE_X\otimes \Omega_X^{[1]}\otimes \Omega_X^{[1]}\xrightarrow{\Id\otimes[\wedge]} \mE_X\otimes \Omega_X^{[2]}$$
    vanishes. Its restriction to $X_\reg$ is a Higgs field in the usual sense. Conversely, a Higgs field on $\mE_X|_{X_\reg}$ need not extend to such a morphism on $X$ when $\mE_X$ is merely reflexive. When $\mE_X$ or $\Omega_X^{[1]}$ is locally free, this extension is automatic and the two definitions coincide.
    
    \vspace{0.1cm}

	Throughout this paper, the Higgs field is defined only on the regular locus. A motivation comes from non-abelian Hodge theory. For a compact K\"ahler klt space $X$, the correspondences established in the subsequent work \cite{ZZZ26} associate representations of $\pi_1(X)$ with locally free Higgs sheaves on $X$, and representations of $\pi_1(X_\reg)$ with reflexive sheaves on $X$ equipped with Higgs fields on $X_\reg$. The latter correspondence motivates the Higgs objects considered here. 
	
    \vspace{0.1cm}

	Non-abelian Hodge theory over projective klt varieties is studied in \cite{GKPT20}, where the underlying sheaves are also considered on the regular locus. In this algebraic setting, a reflexive sheaf on $X_\reg$ admits a unique reflexive extension to $X$. A reflexive coherent analytic sheaf on $X_\reg$, however, need not admit a coherent extension to $X$. We therefore require the underlying coherent sheaf to be defined on the whole analytic space.
\end{remark}

\begin{remark}[The quotient singularities assumption]
	The assumption ``$X$ has quotient singularities in codimension $2$" is essential. On the one hand, it is the most natural setting to introduce orbifold Chern classes and applies to klt spaces. In particular, Theorem \ref{singulargeneralized} generalizes the BG inequality in \cite{GKPT19a} even in the projective setting by weakening the assumption on singularities. On the other hand, it is necessary in the proof of the extension property (Proposition \ref{prop-higgs-orbifold-model}), which is not true in general.
\end{remark}

The approximate critical Hermitian structures of holomorphic vector bundles also give estimates for the maximal and minimal HN slopes under tensor products, symmetric powers, and exterior powers; see Section~\ref{subsec:model-comparison}. These estimates imply that the three operations preserve semistability.
\begin{corollary}\label{semistable-tensorproduct}
	Let $\mE_X$ and $\mF_X$ be torsion-free coherent sheaves on a compact normal K\"ahler space $X$ of dimension $n\geq2$ with quotient singularities in codimension $2$, equipped with Higgs fields $\theta_{\mE_{X_\reg}}$ and $\theta_{\mF_{X_\reg}}$ on $X_\reg$. Let $\alpha_0,\ldots,\alpha_{n-2}$ be K\"ahler classes and set $\Omega=\alpha_0\cdots\alpha_{n-2}$. In the following statements, all tensor products, symmetric powers, and exterior powers are understood as the double duals of their torsion-free quotients.
	\begin{itemize}
		\item[(1)] If $(\mE_X,\theta_{\mE_{X_\reg}})$ is $\Omega$-semistable, then so are $\Lambda^p\mE_X$ for $1\leq p\leq\rank\mE_X$ and $S^p\mE_X$ for $p\geq1$, endowed with the induced Higgs fields on $X_\reg$.
		\item[(2)] If $(\mE_X,\theta_{\mE_{X_\reg}})$ and $(\mF_X,\theta_{\mF_{X_\reg}})$ are $\Omega$-semistable, then so is their tensor product, endowed with the induced Higgs field $\theta_{\mE_{X_\reg}}\otimes\Id_{\mF_{X_\reg}}+\Id_{\mE_{X_\reg}}\otimes\theta_{\mF_{X_\reg}}$ on $X_\reg$.
	\end{itemize}
\end{corollary}

\begin{remark}
    For reflexive sheaves with vanishing Higgs fields, the polystability counterpart of Corollary \ref{semistable-tensorproduct} follows from the Donaldson--Uhlenbeck--Yau (DUY) theorem on compact K\"ahler spaces established in \cite{Chen25}; see also \cite{CGNPPW23,Pan}. A Higgs version of the DUY theorem presents an additional difficulty, since a uniform bound for the Higgs field is not available in general; see the example in \cite{ZZZ26}.
\end{remark}

\subsection{Structure of this paper.}
Section~\ref{preminaries} presents the preliminaries on positivity, complex orbifolds, and homology Chern classes, and establishes the orbifold Chern--Weil formula and HN filtrations for Higgs orbi-sheaves. 

In Section~\ref{higgsorbi-bundle}, we construct $L^p$-approximate critical Hermitian structures on Higgs orbi-bundles over compact Gauduchon orbifolds and prove Corollary~\ref{coro-astheno}.

Section~\ref{higgsheaf-regular} studies coherent sheaves equipped with Higgs fields on the regular locus. We prove that Higgs fields extend to a suitable orbifold modification, study the behavior of HN slopes, and prove Corollary~\ref{semistable-tensorproduct}. 

In Section~\ref{BMY}, we combine these results to prove the BG inequalities on singular K\"ahler spaces, then apply them to the natural Higgs sheaf to obtain the Miyaoka--Yau inequality.

\vspace{0.2cm}

\noindent{\bf Acknowledgements.} The authors thank Wenhao Ou for answering questions about orbifold Chern classes. The second author thanks Masataka Iwai and Satoshi Jinnouchi for discussions on related topics.

\vspace{0.1cm}

\section{Conventions and known results}\label{preminaries}

We collect the definitions and results on complex orbifolds, orbifold Higgs sheaves and orbifold Chern classes used in this paper. Readers familiar with this material may skip this section.

\subsection{Notation}
Throughout this paper, all complex spaces are assumed to be irreducible and all sheaves coherent. The dual of an object $W$ is denoted by $W^\vee$. Objects on the orbifold are distinguished from those on the underlying space by the subscript `orb'. We abbreviate Harder-Narasimhan by ``HN".

For a morphism $f:Y\rightarrow X$ of complex spaces, we write $f^*\mE$ for the ordinary pull-back, $f^{[*]}\mE:=(f^*\mE)^{\vee\vee}$ for the reflexive pull-back, and $f^T\mE:=f^*\mE/\torsion$ for the torsion-free pull-back. The same notation is used chartwise for orbifold morphisms. A saturated transform is always specified as the saturation of the image in the chosen ambient sheaf.

Let $X$ be a complex space of dimension $n$. We refer to \cite{PR94} for related concepts. Denote the K\"ahler differential sheaf of $X$ by $\Omega_X^1$. Set $\Omega_X^p=\Lambda^p\Omega_X^1$, $\mT_X=(\Omega_X^1)^\vee$ and $\Omega_X^{[p]}=(\Omega_X^p)^{\vee\vee}$. For a coherent sheaf $\mF_X$ on $X$, write $\mF_{X_\reg}:=\mF_X|_{X_\reg}$. If $X$ is normal and $\mF_X$ is reflexive, the natural map $\mF_X\to i_*\mF_{X_\reg}$ is an isomorphism, where $i:X_\reg\hookrightarrow X$ is the inclusion. In particular, $\Omega_X^{[p]}=i_*\Omega_{X_\reg}^p$. We use $K_X$ to denote the canonical sheaf $\w_X:=\Omega_X^{[n]}$.

\subsection{$\mathbb{Q}$-line bundles and positivity}\label{complexspaces}

\begin{definition}[$\mathbb{Q}$-line bundle]
	We say that a reflexive sheaf $\mF$ of rank $1$ is a $\mathbb{Q}$-line bundle if $\mF^{[m]}:=(\mF^{\otimes m})^{\vee\vee}$ is locally free.
\end{definition}

\begin{definition}[$\mathbb{Q}$-Gorenstein spaces]
	Let $X$ be a normal space. We say that $X$ is {\em $\mathbb{Q}$-Gorenstein} if $K_X$ is a $\mathbb{Q}$-line bundle. 
\end{definition}

We recall the definitions of the nef, big and K\"ahler cones on normal spaces, following \cite[Section 3]{HP16}. Let $X$ be a reduced complex space. A real-valued function on $X$ is said to be continuous (resp. smooth, pluriharmonic) if it extends to a continuous (resp. smooth, pluriharmonic) function in some local embedding. Smooth forms on $X$ are defined similarly. Let $\mC^\infty_X$ be the sheaf of smooth real-valued functions and $\mathcal{PH}_X$ the sheaf of pluriharmonic functions. We have the exact sequence
$$\mC_X^\infty\rightarrow H^0(X,\mC_X^\infty/\mathcal{PH}_X)\xrightarrow{[\cdot]}H^1(X,\mathcal{PH}_X)\rightarrow 0$$
induced by the short exact sequence $0\rightarrow \mathcal{PH}_X\rightarrow \mC^\infty_X\rightarrow \mC_X^\infty/\mathcal{PH}_X\rightarrow0$, where $[\cdot]$ is the connecting homomorphism in degree $0$.
\begin{definition}[{\cite[Definition 4.6.2]{BG13}}]
	Let $X$ be a complex space. A  $(1,1)$-form with local potentials on $X$ is an element of $H^0(X,\mC_X^\infty/\mathcal{PH}_X)$. The Bott-Chern cohomology is defined by
	$$H^{1,1}_{BC}(X):=H^1(X,\mathcal{PH}_X).$$
\end{definition}
\begin{remark}\label{remark-Qlinebundle-firstchern}
	Since $[\cdot]: H^0(X,\mC_X^\infty/\mathcal{PH}_X)\rightarrow H^{1,1}_{BC}(X,\mathcal{PH}_X)$ is always surjective, an element of $H^{1,1}_{BC}(X)$ can be represented by a closed $(1,1)$-form with local potentials modulo all forms that are globally of the form $dd^cu$. Let $h$ be a smooth Hermitian metric on a holomorphic line bundle $L$. Then $[c_1(L,h)]\in H^{1,1}_{BC}(X)$ is the equivalence class of $\{\varphi_i\}$, where $h=e^{-\varphi_i}$ locally. This is independent of $h$ and is denoted by $c_1(L)$.
	Consider the exact sequence
	$$H^1(X,\R)\rightarrow H^1(X,\mO_X)\rightarrow H^1(X,\mathcal{PH}_X)\xrightarrow{\delta_1} H^2(X,\R)\rightarrow\cdots$$
	induced by the short exact sequence $0\rightarrow\R\rightarrow\mO_X\rightarrow\mathcal{PH}_X\rightarrow0$. As shown in \cite[Proposition 3.5]{GK20}, when $X$ is compact and normal, $H^1(X,\R)\rightarrow H^1(X,\mO_X)$ is surjective. Hence $H_{BC}^{1,1}(X)=H^1(X,\mathcal{PH}_X)\xrightarrow{\delta_1} H^2(X,\R)$ is injective.
\end{remark}

\begin{definition}
	For a $\mathbb{Q}$-line bundle, i.e., $\mL^{[m]}$ is an invertible sheaf for some $m\in\N$, we define its first Chern class as $c_1(\mL):=\frac{1}{m}c_1(\mL^{[m]})\in H^{1,1}_{BC}(X)$, which can be viewed as an element of $H^2(X,\R)$.
\end{definition}

\begin{definition}[K\"ahler forms]
	Let $X$ be a normal complex space. A K\"ahler form $\w$ is a strictly positive closed $(1,1)$-form with local potentials. We say that $X$ is K\"ahler if $X$ admits a K\"ahler metric. We say that a class $\alpha\in H^{1,1}_{BC}(X)$ is K\"ahler if it contains a K\"ahler form.
\end{definition}
A partition of unity gives a Hermitian metric $\w$ on $X$, locally induced by a smooth Hermitian metric on $\C^N$ under an embedding. We can then define nefness on a compact complex space (cf. {\cite[Definition 3]{Pau98}}).

\begin{definition}[Nef and big classes]\label{def-normal-nef}
	Let $X$ be a compact complex space of dimension $n$ with a fixed Hermitian metric $\w$, and let $\eta\in H^{1,1}_{BC}(X)$ be a class represented by a form $u$ with local potentials. 
	\begin{itemize}
		\item[(1)] $\eta$ is called {\em nef} if for every $\epsilon>0$, there exists $f_\epsilon\in\mC_X^\infty$ such that $u+\im\pp f_\epsilon\geq -\epsilon\w$. 
		\item[(2)] We say that a nef class $\eta$ is {\em big} if $\eta^n>0$, i.e., its numerical dimension $v(\eta)=n$.
		\item[(3)] We say that a $\mathbb{Q}$-line bundle $L$ of rank $1$ is {\em nef (resp. big, positive)} if $c_1(L)$ is a nef (resp. big,  K\"ahler) class.
	\end{itemize}
\end{definition}
When no confusion arises, we simply denote $c_1(L)$ by $L$ for any $\mathbb{Q}$-line bundle $L$.
\begin{definition}[Minimal spaces]
	Let $X$ be a compact $\mathbb{Q}$-Gorenstein space. We say that $X$ is minimal if $K_X$ is nef.
\end{definition}

\subsection{Complex orbifolds}\label{orbifolds}
Smooth orbifolds (or V-manifolds) were introduced in \cite{Sa56}. For the study of $L^p$-approximate Hermitian structures on Higgs orbi-bundles, we refer to \cite{F19,DO23,Wu23} for the relevant background on stable orbi-bundles. 
\begin{definition}[Complex orbifolds]
	A complex orbifold $X$ of dimension $n$ is a connected second countable Hausdorff space $X$ equipped with an orbifold structure $X_{\orb}=\{(U_i,G_i,\pi_i)\}$ that satisfies
	\begin{itemize}
		\item[(1)] $U_i$ is an open domain in $\C^n$, $G_i\subset GL_n(\C)$ is a finite group acting holomorphically on $U_i$, $\pi_i$ is the quotient map from $U_i$ to $U_i/G_i$ such that $U_i/G_i\cong X_i$ for some open subset $X_i$ of $X$ and $\bigcup X_i=X$;
		\item[(2)] compatibility conditions: for any two orbifold charts $(U_i,G_i,\pi_i)$ and $(U_j,G_j,\pi_j)$ of open subsets $X_i$ and $X_j$, respectively, and for any $x\in X_i\bigcap X_j$, there is an open neighborhood $X_k$ of $x$ with an orbifold chart $(U_k,G_k,\pi_k)$ such that there are embeddings from $(U_k,G_k,\pi_k)$ to $(U_i,G_i,\pi_i)$ and $(U_j,G_j,\pi_j)$ (an embedding from  $(U_k,G_k,\pi_k)$ to $(U_i,G_i,\pi_i)$ consists of an embedding $\varphi:U_k\rightarrow U_i$ and a group monomorphism $\lambda:G_k\rightarrow G_i$ such that $\varphi$ is equivariant with respect to $\lambda$, i.e., $\varphi(g\cdot x)=\lambda(g)\cdot\varphi(x),\forall g\in G_k$ and $i\circ\pi_k=\pi_i\circ\varphi$ for the inclusion map $i:X_k\hookrightarrow X_i$).
	\end{itemize}
	We say that $X_{\orb}$ is effective if for every $i$, $\ker G=\bigcap_{x\in U_i}\{g\in G_i\ |\ gx=x\}=\{e\}.$ $X_{\orb}$ is called standard if for every $i$, $G_i$ acts freely in codimension $1$. In the latter case, $\pi_i$ is finite and \'{e}tale (smooth and unramified) in codimension $1$. 
\end{definition}

\begin{remark}\label{sot}
	(1) By \cite{Ca57}, the underlying space $X$ of an effective complex orbifold $X_{\orb}$ has a natural structure of a complex space with quotient singularities (which is normal and therefore smooth in codimension $1$) such that the quotient map $\pi:U_i\rightarrow X_i\hookrightarrow X$ is holomorphic.

	(2) Conversely, a complex space $X$ of dimension $n$ with quotient singularities admits a unique standard orbifold structure $\{U_i,\pi_i,G_i\}$ (\cite{Prill67}), and $X_{sing}|_{X_i}=({X_i})_{sing}=\cup_{g\in G_i,g\neq e}\{x\in U_i: g\cdot x=x\}$. Hence, $\pi_i$ is \'etale over $X_{i,reg}$ and $\pi_i^{-1}(X_{sing})$ has codimension at least $2$ in $U_i$ by the definition of the standard orbifold structure, which is essential to our arguments.
\end{remark}

By the preceding remark, a holomorphic function $f$ on any open subset $U\subset X_i$ corresponds to a $G_i$-invariant holomorphic function on $\pi^{-1}(U)\subset U_i$. Hence,
\begin{notation}
	An orbifold subvariety $Z_{\orb}=\{Z_i\}$ of $X_{\orb}$ can be identified with an analytic subvariety $Z$ of $X$, and they have the same codimension. Also, a holomorphism $f_\orb:X_\orb\rightarrow Y_\orb$ of orbifolds can be identified with a holomorphism $f$ between underlying spaces.
\end{notation}

\begin{definition}
	A holomorphic orbi-bundle $E_{\orb}$ over $X_{\orb}$ is a collection of holomorphic $G_i$-linearized vector bundles $\{E_i\}$ over $U_i$ satisfying the following compatibility conditions:
	\begin{itemize}
		\item for any embedding $(\varphi,\lambda):(U_k,G_k,\pi_k)\rightarrow(U_i,G_i,\pi_i)$, there is an isomorphism $\Phi_\varphi:E_k\cong\varphi^* E_i$;
		\item  these isomorphisms are functorial in $\varphi$, namely, for another embedding $(\psi,\mu):(U_i,G_i,\pi_i)\rightarrow(U_j,G_j,\pi_j)$, there holds $\Phi_{\psi\circ\varphi}=(\varphi^*\Phi_\psi)\circ\phi_\varphi$.
	\end{itemize}
\end{definition}

\begin{definition}[Coherent orbi-sheaf]
	A coherent orbi-sheaf $\mE_{\orb}=\{\mE_i\}$ is defined in the same manner. $\mE_{\orb}$ is called torsion-free (resp. reflexive, locally free, torsion) if every $\mE_i$ is torsion-free (resp. reflexive, locally free, torsion).
\end{definition}

\noindent Orbifold differential forms and the cotangent orbi-bundle are defined similarly (see e.g. \cite[Section 2.1]{F19}). We have the de Rham isomorphism theorem for effective complex orbifolds (\cite{Sa56}):
$$H_{dR}^p(X_{\orb},\R)\cong H^p(X,\R),\ H_{c,dR}^p(X_{\orb},\R)\cong H_c^p(X,\R)$$
and the Poincar\'{e} duality $H^p(X,\R)\cong H_c^{2n-p}(X,\R)^\vee$.

\subsubsection{Chern classes of orbi-sheaves}The Chern classes of orbi-bundles are defined using curvature. Let $\mE_{orb}:=\{\mE_i\}$ be a coherent orbi-sheaf of rank $r$ over a complex orbifold $X_{orb}$. Define the determinant line bundle $\det(\mE_i)$ of $\mE_i$ using a finite resolution of $\mE_i$ on each chart $U_i$. Independence of the resolution gives compatibility on overlaps (see \cite[Section 5.3]{KO13}), so $\{\det(\mE_i)\}$ defines a global determinant line bundle.
\begin{definition}[The first Chern class]\label{determinant}
	The determinant line bundle $\det(\mE_{orb})$ of $\mE_{orb}$ is a line orbi-bundle of rank $1$ given by
	$\det(\mE_{orb}):=\{\det(\mE_i)\}.$
	The first Chern class $c_1^\orb(\mE_{orb})$ of $\mE_{orb}$ is defined by
	$c_1^\orb(\mE_{orb}):=c_1^\orb(\det(\mE_{orb}))\in H^2(X,\R).$
\end{definition}

A theory of Chern characters for general coherent orbi-sheaves is developed in \cite{MTTW25} using flat antiholomorphic superconnections.
\begin{definition}[Higher Chern classes of orbi-sheaves, see {\cite{MTTW25}}]
	On a compact complex orbifold, there exists a unique group homomorphism $$\ch^\orb:K(X_\orb)\rightarrow H_{BC}^*(X_\orb,\C)$$ from the Grothendieck group of coherent orbi-sheaves on $X_\orb$ to the Bott–Chern
	cohomology of $X_\orb$ satisfying
	\begin{itemize}
		\item[(1)] $\ch^\orb$ coincides with the definition given by curvature for holomorphic orbi-bundles.
		\item[(2)] $\ch^\orb(Lg_\orb^*\mE_\orb)=g_\orb^*\ch^\orb(\mE_\orb)$, where $Lg_\orb^*$ denotes the derived pull-back in the Grothendieck group. For orbi-bundles this is the ordinary pull-back.
		\item[(3)] $\ch^\orb$ satisfies the Grothendieck-Riemann-Roch formula for embeddings.
	\end{itemize}
\end{definition}

We recall the following elementary statement.
\begin{lemma}[cf. {\cite[Lemma 2.2]{Ou25}}]\label{lem-excision}
	Let $f:Y\rightarrow X$ be a projective bimeromorphism between compact complex spaces of dimension $n$. Suppose that $\sigma_1,\sigma_2\in H_{2n-2k}(Y,\R)$ agree on $Y\setminus f^{-1}(Z)$ for an analytic subvariety $Z\subset X$ of codimension at least $k+1$. Then
	$f_*\sigma_1=f_*\sigma_2\in H_{2n-2k}(X,\R)$.
\end{lemma}

\begin{lemma}\label{lem-compatible-orbifold-pullback}
	Let $g_\orb:Y_\orb\rightarrow X_\orb$ be a bimeromorphism between compact complex orbifolds. Suppose that $\mF_\orb$ on $Y_\orb$ and $\mE_\orb$ on $X_\orb$ correspond over $X\setminus Z$, where $g$ is an isomorphism and $\codim_XZ\geq k+1$. Then
	\[
	 g_*\ch_k^\orb(\mF_\orb)=\ch_k^\orb(\mE_\orb),
	\]
	equivalently $\ch_k^\orb(\mF_\orb)\cdot g^*\gamma=\ch_k^\orb(\mE_\orb)\cdot\gamma$ for every complementary class $\gamma$.
\end{lemma}
\begin{proof}
	Take a common subsheaf of $\mF_\orb$ and the torsion-free pull-back of $\mE_\orb$. Both quotients are supported over $g^{-1}(Z)$. Grothendieck--Riemann--Roch shows that their Chern characters in degrees at most $k$ are supported over $Z$. Lemma~\ref{lem-excision} makes the degree-$k$ push-forwards vanish because $\codim_XZ\geq k+1$. Functoriality on $X\setminus Z$ then gives the stated equality.
\end{proof}

\subsubsection{Chern-Weil formula}
For a saturated (i.e. $E_\orb/\mF_\orb$ is torsion-free) orbi-subsheaf $\mF_\orb$ of a holomorphic orbi-bundle $E_\orb$ on a compact complex orbifold, the singular set $\Sigma_\orb:=S_{n-1}(\mF_\orb)\cup S_{n-1}(E_\orb/\mF_\orb)$ has codimension at least $2$. Then $\mF_\orb|_{X_\orb\setminus\Sigma_\orb}$ is a subbundle of $E_\orb|_{X_\orb\setminus\Sigma_\orb}$. If $H_\orb$ is a Hermitian metric on $E_\orb$ and $\eta_\orb$ is a smooth $\partial\bar\partial$-closed $(n-1,n-1)$-form (equivalently, an Aeppli class), the Chern--Weil formula reads
\begin{equation}\label{chern-weil}
	\begin{split}
		&c_1^\orb(\mF_\orb)\cdot[\eta_\orb]=\int_{X_\orb\setminus\Sigma_\orb}\frac{\im}{2\pi}\tr F_{H_{\mF_\orb}}\wedge{\eta_\orb}\\
		=&\frac{1}{2\pi}\int_{X_\orb\setminus\Sigma_\orb}\left(\im\tr(\pi_{\mF_\orb}^{H_\orb}\circ F^{H_\orb})-\im\tr(\p\pi_{\mF_\orb}^{H_\orb}\wedge \bp\pi_{\mF_\orb}^{H_\orb})\right)\wedge{\eta_\orb}
	\end{split}
\end{equation}
where $H_{\mF_\orb}$ is the induced metric on $\mF_\orb|_{X_\orb\setminus \Sigma_\orb}$ and $\pi_{\mF_\orb}^{H_\orb}$ is the orthogonal projection onto $\mF_\orb$ with respect to the metric $H_\orb$. This formula was proved on Gauduchon manifolds in \cite{Br01}. We give a direct proof using the following orbifold resolution lemma.

\vspace{0.2cm}

The lemma follows from functorial resolution of singularities.
\begin{lemma}\label{lem-resoorbi}
	Let $\mE_{\orb}$ be a torsion-free coherent orbi-sheaf on a complex orbifold $Y_{\orb}$. Then there exists a projective bimeromorphism $g$ from a complex orbifold $W_{\orb}$ to $Y_{\orb}$ satisfying
	\begin{itemize}
		\item[(1)] $W_{\orb}$ admits an orbifold structure $\{\widehat{V}_i,G_i\},$ such that $g=\{g_i:\widehat{V}_i\rightarrow V_i\}$, each $g_i$ being a composition of blowups with $G_i$-invariant smooth centers contained in the non-locally-free locus of $\mE_{\orb}$.
		\item[(2)] The torsion-free pull-back $E_\orb=g_{\orb}^T\mE_{\orb}:=\{g_i^*\mE_i/(\mathrm{tor})\}$ is a vector orbi-bundle.
	\end{itemize}
\end{lemma}

\begin{proof}[Proof of \eqref{chern-weil}]
  Apply Lemma \ref{lem-resoorbi} successively to $E_\orb/\mF_\orb$ and $\mF_\orb$. This gives an orbifold $Y_\orb$, a bimeromorphism $f_\orb:Y_\orb\rightarrow X_\orb$, and a $f_\orb^*\theta_{\orb}$-invariant orbi-subbundle $\mF_\orb'$ of $f_\orb^*\mE_\orb$ with $(f_\orb)_*{\mF_\orb'}=\mF_\orb$ in codimension $1$. Compute $c_1^\orb(\mF_\orb')$ using the induced Hermitian metric $\pi_\orb^*H_\orb|_{\mF_\orb'}$ and apply Lemma \ref{lem-compatible-orbifold-pullback} to obtain \eqref{chern-weil}.
\end{proof}

\subsubsection{HN filtration of Higgs orbi-sheaves}
Let $X$ be a compact complex space with quotient singularities and $X_{orb}:=\{U_i,\pi_i,G_i\}$ be a complex orbifold structure of $X$.
\begin{definition}[Higgs orbi-sheaves]
	A {\em Higgs orbi-sheaf} $(\mE_{orb},\theta_{orb}):=\{(\mE_i,\theta_i)\}$ over $X_{orb}$ consists of a coherent orbi-sheaf $\mE_{orb}$ and a morphism $\theta_{orb}:\mE_{orb}\rightarrow\mE_{orb}\otimes\Omega_{X_{orb}}$ such that $\theta_{orb}\wedge\theta_{orb}=0$. We say that $(\mE_{orb},\theta_{orb})$ is torsion-free (resp. reflexive, locally free) if $\mE_{orb}$ is torsion-free (resp. reflexive, locally free). An orbi-subsheaf $\mF_{orb}$ is called $\theta_{orb}$-invariant if $\theta_{orb}(\mF_{orb})\subseteq \mF_{orb}\otimes\Omega_{X_{orb}}$, i.e., $\theta_i(\mF_i)\subseteq\mF_i\otimes\Omega_{U_i}$, $\forall i$.
\end{definition}

\begin{definition}[Nef $(p,p)$-classes]
	$\gamma\in H_{BC}^{p,p}(X_\orb,\R)$ is said to be {\em nef} if for any $\epsilon>0$, there exists some smooth orbifold $(p,p)$-form $\eta_{\orb,\epsilon}\in \gamma$ such that $\eta_\epsilon\geq-\epsilon \w_\orb^{p}$, where $\w_\orb$ is a fixed Hermitian metric.
\end{definition}

As in the smooth case, the degree of a coherent sheaf $\mE_{orb}$ with respect to some orbifold nef $(n-1,n-1)$-class $\eta_\orb$ is defined by $\deg_{\eta_{orb}}(\mE_{orb}):=c_1^\orb(\mE_{orb})\cdot[\eta_{orb}],$ and the slope  by
$\mu_{\eta_{orb}}(\mE_{orb}):=\frac{\deg_{\eta_{orb}}(\mE_{orb})}{\rank \mE_{orb}}.$ Stability is defined as follows.

\begin{definition}[Semistability]
	A torsion-free Higgs orbi-sheaf $\mE_{orb}$ is said to be {\em $\eta_{orb}$-semistable} if $\mu_{\eta_{orb}}(\mF_{orb})\leq\mu_{\eta_{orb}}(\mE_{orb})$ for every $\theta_{orb}$-invariant orbi-subsheaf $\mF_{orb}\subseteq\mE_{orb}$ with $0<\rank\mF_{orb}<\rank\mE_{orb}$. It is {\em stable} if all these inequalities are strict.
\end{definition}

For the remainder of this subsection, let $\eta_\orb=\omega^{n-1}/(n-1)!$, where $\omega$ is a Gauduchon orbifold metric. This is the only case needed in Section~\ref{higgsorbi-bundle}. The positivity of $\eta_\orb$ is essential in the boundedness argument below; we do not claim the corresponding existence statement for an arbitrary degenerate nef $(n-1,n-1)$-class.
A standard argument of \cite[Proposition 5.6.14]{Koba} gives
\begin{lemma}\label{lem-orbi-torsionsheaf}
	For any torsion orbi-sheaf $\mT_\orb$, $c_1^\orb(\mT_\orb)\cdot\eta_\orb\geq0$. In particular, for an orbi-subsheaf $\mF_\orb$, we have $c_1^\orb(\mF_\orb)\cdot[\eta_\orb]\leq c_1^\orb(\mathcal{H}_\orb)\cdot[\eta_\orb]$, where $\mathcal{H}_\orb$ is the saturation of $\mF_\orb$ in $\mE_\orb$.
\end{lemma}

\noindent Combining this with \eqref{chern-weil}, we obtain
\begin{lemma}\label{boundness}
	There exists a constant $C$ such that every coherent orbi-subsheaf $\mF_\orb\subset\mE_\orb$ satisfies $\deg_{\eta_\orb}(\mF_\orb)\leq C$.
\end{lemma}

\begin{proof}
	By Lemma~\ref{lem-orbi-torsionsheaf}, we may assume that $\mF_\orb$ is saturated. Take a fixed modification $p_\orb:Y_\orb\to X_\orb$ such that $E_\orb=p_\orb^T\mE_\orb$ is locally free, and let $\mF'_\orb\subset E_\orb$ be the saturated transform of $\mF_\orb$. Put $\eta'_\orb=p_\orb^*\eta_\orb$, and choose a smooth Hermitian metric $H$ on $E_\orb$ and a Hermitian form $\omega_Y$ on $Y_\orb$. The form $\eta'_\orb$ is semipositive and $\partial\bar\partial$-closed. Exceptional divisors pair trivially with this pulled-back form by the projection formula. Thus \eqref{chern-weil} gives
	\[
	\begin{split}
	 \deg_{\eta_\orb}(\mF_\orb)
	 &=\frac1{2\pi}\int_{Y_\orb\setminus\Sigma}
	 \left(\im\tr(\pi_{\mF'}F_H)
	 -\im\tr(\partial\pi_{\mF'}\wedge\bar\partial\pi_{\mF'})\right)\wedge\eta'_\orb\\
	 &\leq C\int_{Y_\orb}\omega_Y\wedge\eta'_\orb.
	\end{split}
	\]
	Here $\pi_{\mF'}$ is the $H$-orthogonal projection on the subbundle locus and $\Sigma$ is its complement. The second fundamental form contributes a nonpositive term, and the ambient curvature is bounded by the fixed form $C\omega_Y$. The last integral is independent of $\mF_\orb$.
\end{proof}

The saturation of a $\theta_\orb$-invariant orbi-sheaf is $\theta_\orb$-invariant by the following lemma.
\begin{lemma}\label{lem-saturation-invariant}
	Let $(\mE,\theta)$ be a torsion-free Higgs sheaf on a complex manifold $V$. Suppose that $\mF$ is a subsheaf of $\mE$ and that there exists a Zariski dense open subset $W\subset V$ such that $\mF|_W$ is $\theta$-invariant. Then its saturation $\mF^{sat}$ is $\theta$-invariant.
\end{lemma}

\begin{proof}
	Consider the composition
	$$\mF^{sat}\xrightarrow{\theta|_{\mF^{sat}}}\mE\otimes\Omega_V^1\xrightarrow{projection}(\mE/\mF^{sat})\otimes\Omega_V^1,$$
	which vanishes on $W\setminus\Sigma$ where $\Sigma$ is the non-locally-free locus of $\mE/\mF$, and thus vanishes on $V$ since $(\mE/\mF^{sat})\otimes\Omega_V^1$ is torsion-free.
\end{proof}

\begin{lemma}[The saturated orbi-subsheaf with maximal slope]\label{maximalslope}
	Let $(E_\orb,\theta_\orb)$ be a torsion-free Higgs orbi-sheaf polarised by the Gauduchon class $\eta_\orb=\omega^{n-1}/(n-1)!$. There is a unique $\theta_\orb$-invariant saturated orbi-subsheaf $\mE_{1,\orb}$ of maximal slope and, among those of maximal slope, maximal rank. In particular, $\mE_{1,\orb}$ is $\eta_\orb$-semistable with the induced Higgs field.
\end{lemma}

\begin{proof}
	Using Lemma \ref{lem-saturation-invariant}, Lemma \ref{lem-orbi-torsionsheaf} and Lemma \ref{boundness}, the argument of \cite[Pages 82-84]{GKKP14} gives the existence of $\mE_{1,\orb}$. Uniqueness follows from the argument of \cite[Pages 591]{Br01}.
\end{proof}

By applying Lemma \ref{maximalslope} to $\mE_\orb/\mE_{1,\orb}$ and proceeding by induction, we obtain
\begin{lemma}[HN filtration]\label{lem-orbi-HNfiltration}
		Let $(\mE_\orb,\theta_\orb)$ be a torsion-free Higgs orbi-sheaf of rank $r$ over a compact complex orbifold, polarised by $\eta_\orb=\omega^{n-1}/(n-1)!$ for a Gauduchon orbifold metric $\omega$. Then there exists a unique filtration by invariant coherent orbi-subsheaves
	$$0=\mE_{\orb,0}\subset\mE_{\orb,1}\subset\cdots\subset\mE_{\orb,l}=\mE_\orb,$$
		such that every torsion-free quotient
		\[
		 \mQ_{\orb,k}=\mE_{\orb,k}/\mE_{\orb,k-1}
		\]
		is $\eta_\orb$-semistable and its slopes satisfy
		$\mu_{\eta_\orb}(\mQ_{\orb,1})>\cdots>\mu_{\eta_\orb}(\mQ_{\orb,l})$.
\end{lemma}

\begin{definition}[HN type]
	In the notation of Lemma \ref{lem-orbi-HNfiltration}, the {\em HN type} is defined by $$\vec{\mu}_{\eta_\orb}(\mE_\orb,\theta_\orb):=(\mu_{1,\eta_\orb},\cdots,\mu_{r,\eta_\orb}),$$
	where $\mu_{i,\eta_\orb}:=\mu_{\eta_\orb}(\mQ_{\orb,k})$ for $\rank(\mE_{\orb,k-1})+1\leq i\leq \rank (\mE_{\orb,k})$.
\end{definition}
\noindent In particular,
\begin{equation}\label{the minimal slope}
	\mu_{r,\eta_\orb}=\inf\{\mu_{\eta_\orb}(\mathcal{Q}_{\orb})|\mathcal{Q}_\orb\text{ is a $\theta_\orb$-invariant quotient orbi-sheaf}\},
\end{equation}
 \begin{equation}\label{the maximal slope}
 	\mu_{1,\eta_\orb}=\sup\{\mu_{\eta_\orb}(\mathcal{S}_\orb)|\mathcal{S}_\orb \text{ is a } \theta_{\orb}\text{-invariant torsion-free orbi-subsheaf}\}.
 \end{equation}

\subsection{Homology Chern classes}\label{chernclasses}

\subsubsection{The homology first Chern class}
\begin{definition}\label{defn-firstchern}
	Let $X$ be a compact normal complex space of dimension $n$ and $\mE$ a coherent sheaf on $X$. Choose a resolution of singularities $f:\widehat X\to X$. The homology first Chern class $c_1(\mE)\in H_{2n-2}(X,\R)$ is defined by
	$$c_1(\mE)\cdot\gamma:=c_1(f^*\mE)\cdot f^*\gamma,\qquad
	\gamma\in H^{2n-2}(X,\R).$$
	Here $c_1(f^*\mE):=c_1(\det(f^*\mE))$ is the usual first Chern class of the coherent sheaf $f^*\mE$ on the smooth space $\widehat X$.
\end{definition}

The definition is independent of the resolution and gives the following property (cf. \cite[Proposition 5.6.14]{Koba}).
\begin{lemma}\label{lemma-slope-torsion}
	Let $\mT$ be a torsion sheaf on a compact normal space $X$ of dimension $n$. Then for any nef classes $\eta_0,\cdots,\eta_{n-2}$,
	$c_1(\mT)\cdot \eta_0\cdots\eta_{n-2}\geq0.$
\end{lemma}

The same homology restriction argument proves the following compatibility.

\begin{lemma}\label{lemma-compatible-pullback}
	Let $f:Y\rightarrow X$ be a bimeromorphism between compact normal spaces. If $\mF$ and $\mE$ are coherent sheaves on $Y$ and $X$, respectively, such that $f_*\mF=\mE$ in codimension $1$, then $f_*c_1(\mF)=c_1(\mE)$.
\end{lemma}

\begin{proof}
	The two classes $f_*c_1(\mF)$ and $c_1(\mE)$ agree outside an analytic subset of $X$ of codimension at least $2$. Their difference is supported on this subset and vanishes by the homology restriction sequence, as in Lemma~\ref{lem-excision}.
\end{proof}

As a consequence, we obtain the following elementary property.
\begin{lemma}[Additivity]\label{lem-aditivity}
	For any short exact sequence $0\rightarrow\mF\rightarrow\mE\rightarrow\mQ\rightarrow0$ on a compact normal space, $c_1(\mE)=c_1(\mF)+c_1(\mQ)$.
\end{lemma}

\begin{proof}
	At the generic point of each divisor, the local ring is a discrete valuation ring, and the determinant and torsion-length formulas are additive in a short exact sequence. The asserted identity therefore holds outside a set of codimension at least $2$. Lemma~\ref{lem-excision} gives the equality on $X$.
\end{proof}

\subsubsection{Push-forward of orbi-sheaves to the quotient space}
We recall the $G$-invariant push-forward of sheaves, following \cite[Appendix A]{GKKP11}.
\begin{definition}
	A $G$-sheaf $\mE$ on $X$ is a coherent sheaf of $\mO_X$-modules such that for any open subset $U\subseteq X$ and any $g\in G$, there exist natural push-forward morphisms $(\phi_g)_*:\mE(U)\rightarrow\mE(\phi_g(U))$ that satisfy the usual compatibility conditions for sheaves, where $\phi_g$ denotes the automorphism associated with $g$.
\end{definition}

\begin{definition}[$G$-invariant sheaf]
	If $G$ acts trivially on $X$, and if $\mE$ is a $G$-sheaf, the associated sheaf of invariants, denoted $\mE^G$, is the sheaf defined by
	$$\mE^G(U):=(\mE(U))^G=\{s\in\mE(U)|(\phi_g)_*s=s,\ \forall g\in G\},$$
	where $(\mE(U))^G$ denotes the submodule of $G$-invariant elements of $\mE(U)$.
\end{definition}
\begin{definition}[$G$-invariant pushforward]
	Let $\mE$ be a $G$-sheaf on $X$. Let $\pi:X\rightarrow X/G$ be the quotient morphism. Then $G$ acts trivially on $X/G$ and the pushforward $\pi_*\mE$ admits a natural $G$-sheaf structure on $X/G$. The $G$-invariant pushforward of $\mE$ is defined by $(\pi_*\mE)^G$.
\end{definition}

We will use the following properties. Let $X_{orb}:=\{V_i,G_i,\pi_i\}$ be a compact complex orbifold and $X$ its underlying quotient space. For a coherent sheaf $\mE_X$ on $X$, $\{\pi_i^*\mE\}$ defines a coherent orbi-sheaf on $X_\orb$. Conversely, we have the following.
\begin{lemma}[$G_i$-invariant push-forward]\label{push-forward}
	Let $\mE_\orb$ be an orbi-sheaf on $X_\orb$. If $\mE_\orb$ is torsion-free (resp. reflexive), then $\mE:=\left((\pi_i)_*\mE_i\right)^{G_i}$ is a torsion-free (resp. reflexive) sheaf on $X$. If $0\rightarrow \mF_\orb\rightarrow\mE_\orb\rightarrow\mQ_\orb\rightarrow$ is a short exact sequence of orbi-sheaves, $0\rightarrow \left((\pi_i)_*\mF_i\right)^{G_i}\rightarrow \left((\pi_i)_*\mE_i\right)^{G_i}\rightarrow \left((\pi_i)_*\mS_i\right)^{G_i}\rightarrow$ is also exact.
\end{lemma}

The well-definedness follows from the compatibility conditions for orbi-sheaves. Lemma \ref{push-forward} then follows by applying {\cite[Lemma A.3 and Lemma A.4]{GKKP11}} to each $i$.

\begin{lemma}[Pull-back of sheaves to an orbifold structure]\label{lift}
	Let $\mE$ be a torsion-free coherent sheaf on $X$, and set $\mG_\orb=\{\pi_i^*\mE/\torsion\}$ and $\mF_\orb=\{(\pi_i^*\mE)^{\vee\vee}\}$. Then
	$c_1(\mE)=c_1^\orb(\mG_\orb)$ and $c_1(\mE)=c_1^\orb(\mF_\orb)$.
\end{lemma}

\begin{proof}
	Set $L=(\det\mE)^{\vee\vee}$. Since $X$ has quotient singularities, $L$ is a $\mathbb Q$-line bundle. Choose a positive integer $m$ such that $L^{[m]}$ is a line bundle. On each chart, the line bundles
	$$(\det\mG_i)^{\otimes m}\quad\text{and}\quad\pi_i^*L^{[m]}$$
	agree outside a set of codimension at least $2$, where $\mE$ is locally free. They therefore agree everywhere. The same holds for $\mF_i$. Taking first Chern classes and dividing by $m$ proves both identities.
\end{proof}

Combining Lemma \ref{lift} with Step 1 in the proof of \cite[Lemma 3.15, pp.~23--24]{DO23}, we obtain the following.
\begin{lemma}[$G_i$-invariant push-forward]\label{descent}
	Let $q$ be the least common multiple of the orders of the local groups in a finite orbifold atlas of $X_\orb$. Let $\mE_\orb$ be a torsion-free orbi-sheaf on $X_\orb$. Set $\mE:=\left((\pi_i)_*\mE_i\right)^{G_i}$ and $\mG_\orb:=\{\pi_i^*\mE/(\torsion)\}$. Then $\mE$ is torsion-free and 
	$$c_1^\orb(\mE_\orb)=c_1(\mE)+\sum\limits_{j=1}^s d_jc_1^\orb(E_{\orb,j})$$
	where $E_{\orb,j},\ j=1,\cdots,s$ are the irreducible components of the ramification locus of ${\pi_i}$ and $d_j$ is the vanishing order of the natural morphism $\det (\mG_\orb)\rightarrow\det(\mE_\orb)$. Then $d_j$ is an integer satisfying $0\leq d_j\leq(\rank\mE_\orb)(q-1)$. In particular, $c_1^\orb(\mE_\orb)=c_1(\mE)$ in $H_{2n-2}(X,\R)$ when $X_\orb$ is standard.
\end{lemma}

\subsubsection{Orbifold Chern classes}
We use partial orbifold resolutions to define Chern classes on spaces with quotient singularities in codimension $2$. This includes klt spaces by the following result.
\begin{lemma}[{\cite[Lemma 5.8]{GKPT20}}]
	Let $X$ be a klt space. Then there exists a closed analytic subset $Z$ of codimension at least $3$ in $X$ such that $X\setminus Z$ has only quotient singularities, i.e., $\forall x\in X\setminus Z,$ there exists some finite group $G\subset GL(n,\C)$ such that $(X,x)\cong(\C^n/G,0)$ as germs of complex spaces.
\end{lemma}

\begin{lemma}[cf. {\cite[Theorem 1]{KO25}}]\label{lemma-existence}
	Let $X$ be a compact complex space. Assume that $X$ has quotient singularities in codimension $2$. Then there exists a projective bimeromorphic morphism $f:Y\rightarrow X$ such that $Y$ has quotient singularities, and the indeterminacy locus of $f^{-1}$ has codimension at least $3$ in $X$. Such a morphism will be referred to as a {\em partial orbifold resolution} of $X$.
\end{lemma}

\begin{definition}[Orbifold first and second Chern classes]\label{defn-orbifold}
	Let $X$ be a compact normal complex space of dimension $n\geq2$ with quotient singularities in codimension $2$, and let $\mE$ be a reflexive sheaf on $X$. Choose a partial orbifold resolution $f:Y\rightarrow X$, with standard orbifold structure $Y_\orb=\{V_i,G_i,\pi_i\}$, and put
	$$\mE_\orb=f_\orb^{[*]}\mE:=\{((f\circ\pi_i)^*\mE)^{\vee\vee}\}.$$
	For $i=1,2$, define $\widehat c_i(\mE)\in H_{2n-2i}(X,\R)$ by
	$$\widehat c_i(\mE)\cdot\sigma
	:=c_i^\orb(\mE_\orb)\cdot f^*\sigma,\qquad
	\sigma\in H^{2n-2i}(X,\R).$$
	Similarly, define $\widehat c_1(\mE)^2\in H_{2n-4}(X,\R)$ by
	$$\widehat c_1(\mE)^2\cdot\sigma
	:=c_1^\orb(\mE_\orb)^2\cdot f^*\sigma.$$
	The square is taken on the orbifold before push-forward. For $r=\rank\mE$, put
	$$\widehat\Delta(\mE):=2\widehat c_2(\mE)-\frac{r-1}{r}\widehat c_1(\mE)^2.$$
\end{definition}

We explain why this definition is independent of $f$. Enlarge the exceptional image to an analytic set $Z\subset X$ of codimension at least $3$ such that $\mE|_{X\setminus Z}$ induces a vector bundle on the standard orbifold structure. This is possible because a reflexive sheaf on a smooth chart is locally free in codimension $2$. Any two choices give the same classes over $X\setminus Z$. For $k=2n-2,2n-4$, the exact sequence
$$H_k(Z,\R)\longrightarrow H_k(X,\R)\longrightarrow H_k^{\mathrm{BM}}(X\setminus Z,\R)$$
has zero first term, since $\dim_\R Z\leq2n-6$. The pushed-forward classes are therefore equal. This is the construction of \cite[Definition 2.1 and Proposition 2.3]{Ou25}, which applies to quotient singularities in codimension $2$.

Let $g_\orb:W_\orb\to Y_\orb$ be a sheaf resolution of $\mE_\orb$ as in Lemma~\ref{lem-resoorbi}, and put $E_\orb=g_\orb^T\mE_\orb$. All its centers lie over $Z$, so the same argument gives
\begin{equation}\label{equa-orbi-computation}
	\widehat c_2(\mE)\cdot\sigma
	=c_2^\orb(E_\orb)\cdot(f\circ g)^*\sigma,\qquad
	\sigma\in H^{2n-4}(X,\R).
\end{equation}
The analogous identities hold for $\widehat c_1$ and $\widehat c_1^2$. In particular, $\widehat c_1(\mE)=c_1(\mE)$.

\begin{lemma}[Calculus of orbifold Chern classes]\label{calculus}
	Let $X$ be a compact normal complex space with quotient singularities in codimension $2$, and let $\mE,\mF$ be reflexive sheaves on $X$. Then, for $i=1,2$,
	$$\widehat c_i(\mE^\vee)=(-1)^i\widehat c_i(\mE),\qquad
	\widehat c_1(\mE\oplus\mF)=\widehat c_1(\mE)+\widehat c_1(\mF),$$
	$$\widehat c_1(\mE^\vee)^2=\widehat c_1(\mE)^2,\qquad
	\widehat c_2(\mE\oplus\mO_X)=\widehat c_2(\mE),$$
	and
	$$\widehat\Delta(\End(\mE))=2\rank(\mE)\,\widehat\Delta(\mE).$$
\end{lemma}

\begin{proof}
	Choose $Z\subset X$ of codimension at least $3$ such that both sheaves induce vector bundles on the standard orbifold over $X\setminus Z$. There all the asserted identities are the usual vector-bundle identities. In particular, writing $r=\rank\mE$, we have
	$$c_1(\End E)=0,\qquad
	\ch_2(\End E)=(r-1)c_1(E)^2-2r c_2(E)=-r\Delta(E),$$
	so $\Delta(\End E)=2r\Delta(E)$. The injectivity of the homology restriction used after Definition~\ref{defn-orbifold} extends these identities to $X$.
\end{proof}

\section{$L^p$-approximate critical Hermitian structure}\label{higgsorbi-bundle}
The main result of this section is the existence of an $L^p$-approximate Hermitian–Einstein structure on Higgs orbi-bundles over a Gauduchon orbifold $(X_\orb,\w_\orb)$, which has not been treated even in the smooth case. We follow the argument of \cite{LZZ21}, which treats the case $\theta=0$.

\begin{notation}\label{nota-rescale}
	Since all objects in this section are orbifold objects, we omit the subscript `orb' from the notation. To simplify the coefficients in the estimates, we rescale the slope $\mu_{\w}(\mF)$ by $\frac{1}{(n-1)!}\mu_{\w}(\mF)$ for any orbi-sheaf $\mF$.
\end{notation}

Let $H$ be a Hermitian metric on a Higgs orbi-bundle $E$ over a Gauduchon orbifold $(X,\w)$, where $\w$ is a smooth orbifold $(1,1)$-form satisfying $\pp\w^{n-1}=0$. The Hitchin-Simpson connection (\cite{Simpson88}) is defined by
\begin{equation*}
	\bar{\partial}_{\theta}:=\bar{\partial}_{E}+\theta , \quad D_{H,  \theta }^{1, 0}:=\partial_H  +\theta^{* H}, \quad D_{H,  \theta }= \bar{\partial}_{\theta}+ D_{H,  \theta }^{1, 0},
\end{equation*}
where $\partial_H$ is the $(1, 0)$-part of the Chern connection $D_{H}$ of $(E,\bar{\partial }_{E}, H)$ and $\theta^{* H}$ is the adjoint of $\theta $ with respect to $H$.
The curvature of the Hitchin-Simpson connection is
\begin{equation*}
	F_{H,\theta}=F_H+[\theta,\theta^{* H}]+\partial_H\theta+\bar{\partial}_E\theta^{* H},
\end{equation*}
where $F_H$ is the curvature of $D_{H}$. 
\begin{definition}[Hermitian-Einstein structure]
	We say that $H$  is
	a Hermitian-Einstein metric on the Higgs orbi-bundle $(E, \bar{\partial }_{E}, \theta  )$ if
$
		\sqrt{-1}\Lambda_{\omega} F_{H,\theta}
		=\lambda\cdot \mathrm{Id}_{E},
$
	where  $\Lambda_{\omega }$ denotes the contraction with  $\omega $, and $\lambda =\frac{2\pi}{\Vol(X,\w)}\mu_{\omega}(E) $.
\end{definition}
\subsection{Main results}\label{Lpapproximate} We study the perturbed Hermitian-Einstein equation on $(E, \bar{\partial }_{E}, \theta )$ as in \cite{UY86}:
\begin{equation} \label{eq}
	\sqrt{-1}\Lambda_{\omega } (F_{H}+[\theta, \theta^{\ast H}])-\lambda \cdot \textmd{Id}_E+\varepsilon \log (K^{-1}H)=0,
\end{equation}
where $K$ is a fixed background metric. For every $\varepsilon\in(0,1]$,
the continuity method of \cite{LY,NZ} produces a solution $H_\varepsilon$.
The openness argument, maximum principle, closedness estimates, and elliptic
regularity are local and equivariant, so they apply on each finite orbifold
chart and patch globally. If the Higgs orbi-bundle is stable, the solutions
admit a uniform $C^0$ bound and converge, after bootstrapping, to a
Hermitian--Einstein metric.
\begin{theorem}[Donaldson-Uhlenbeck-Yau Theorem]\label{thm1}
	Let $(X, \omega)$ be a  Gauduchon orbifold and $(E, \bar{\partial }_{E}, \theta )$ be a Higgs orbi-bundle on $X$. If $(E, \bar{\partial }_{E}, \theta )$ is $\omega$-stable, then there is a Hermitian-Einstein metric on $(E, \bar{\partial }_{E}, \theta )$. If $(E, \bar{\partial }_{E}, \theta )$ is $\omega$-semistable, then there is an approximate Hermitian-Einstein metric on $(E, \bar{\partial }_{E}, \theta )$.
\end{theorem}
If $(E, \bar{\partial }_{E}, \theta )$ is not $\omega$-stable, a uniform $C^{0}$-estimate for $H_{\varepsilon}$ with $\varepsilon \in (0, 1]$ may fail. We can nevertheless study the limiting behavior of the solutions $H_{\varepsilon}$ of the perturbed equation (\ref{eq}) as $\varepsilon \rightarrow 0$. Consider the
Harder-Narasimhan filtration of $(E, \bar{\partial}_{E}, \theta )$ with respect to $\w$ constructed in Lemma \ref{lem-orbi-HNfiltration}. For each $\mathcal{E}_{\alpha }$ and the Hermitian metric $K$,  we have the associated orthogonal
projection $\pi_{\alpha }^{K}:E\rightarrow E$ onto $\mathcal{E}_{\alpha }$ with respect to  $K$. Each
$\pi_{\alpha }^{K}$ is an $L_{1}^{2}$-bounded  Hermitian endomorphism. We therefore define an $L_{1}^{2}$-bounded Hermitian endomorphism by
\begin{equation}\Phi_{\omega}^{HN} (E, \theta , K)=\Sigma_{\alpha=1}^{l}\mu_{\omega } (\mathcal{Q}_{\alpha })(\pi_{\alpha }^{K}-\pi_{\alpha-1}^{K}),\end{equation} which is called the
Harder-Narasimhan projection of  the Higgs orbi-bundle $(E, \bar{\partial}_{E}, \theta )$. Denote the $r$ eigenvalues of the mean curvature $\sqrt{-1}\Lambda_{\omega} F_{H, \theta }$  by $\lambda_1(H, \theta , \omega)$, $\lambda_2(H, \theta , \omega)$, $\cdots$, $\lambda_r(H, \theta , \omega)$, in decreasing order. Then each $\lambda_{\alpha }(H, \theta , \omega)$ is Lipschitz continuous.
Set
\begin{equation}\vec\lambda(H, \theta ,\omega)=(\lambda_1(H, \theta ,\omega),\lambda_2(H, \theta ,\omega),\cdots,\lambda_r(H,\theta ,\omega)),
\end{equation}
and
$$
		\lambda_{mU}(H, \theta , \omega)=\frac{1}{\Vol(X,\w)}\int_X\lambda_1(H, \theta , \omega)\frac{\omega^n}{n!},\ \lambda_{mL}(H, \theta , \omega)=\frac{1}{\Vol(X,\w)}\int_X\lambda_r(H, \theta , \omega)\frac{\omega^n}{n!}.
$$

Following the argument of \cite{LZZ21} for $\theta=0$, we obtain an
$L^{p}$-approximate critical Hermitian structure on the Higgs orbi-bundle $(E, \bar{\partial }_{E},\theta)$.

\begin{theorem}[$L^p$-approximate critical Hermitian structure]\label{thm2}
	Let $(X, \omega )$ be a compact Gauduchon orbifold of complex dimension $n$,  $(E,\bar{\partial}_E, \theta)$ be a  Higgs orbi-bundle of rank $r$ over $X$, $K$ be a fixed Hermitian metric on $E$ and $H_{\varepsilon}$ be a solution of the perturbed equation (\ref{eq}).  Then there exists a sequence $\varepsilon_{i}\rightarrow 0$ such that
	\begin{equation}\label{eqthm1}
		\lim_{i\rightarrow \infty }\left\| \sqrt{-1}\Lambda_{\omega }F_{H_{\varepsilon_{i}}, \theta }-\frac{2\pi }{\Vol(X,\w)}\Phi_{\omega }^{HN}(E,\theta , K)\right\|_{L^{p}(K)}=0
	\end{equation}
	for any $0<p<+\infty $.
\end{theorem}

Recall \eqref{the maximal slope}, \eqref{the minimal slope}, and the
Chern--Weil formula \eqref{chern-weil}.  Theorem~\ref{thm2} implies the
following integrated characterization. Finite-$p$ convergence gives this
averaged statement, without a pointwise $L^\infty$ conclusion.
\begin{corollary}\label{coro-characterization}
Let $\lambda_{mU}$ and $\lambda_{mL}$ be the averaged largest and smallest
eigenvalues defined above. Then
\[
 \frac{2\pi}{\Vol(X,\omega)}\mu_{1,\omega}
   =\inf_H\lambda_{mU}(H,\theta,\omega),\qquad
 \frac{2\pi}{\Vol(X,\omega)}\mu_{r,\omega}
   =\sup_H\lambda_{mL}(H,\theta,\omega),
\]
where the infimum and supremum run over all smooth Hermitian metrics on $E$.
\end{corollary}

We now derive the quantitative mixed Bogomolov inequality.  In this
statement the slopes are unnormalised, namely
\[
 \mu_\Omega(\mF)=
 \frac{c_1^\orb(\mF)\cdot[\omega_0]\cdots[\omega_{n-2}]}
      {\rank\mF}.
\]

\begin{corollary}[Quantitative mixed Bogomolov inequality]\label{coro-astheno}
Let $\omega_0,\ldots,\omega_{n-2}$ be orbifold K\"ahler forms and set
\[
 \Omega=[\omega_0]\cdots[\omega_{n-2}],\qquad
 \Theta=\omega_1\wedge\cdots\wedge\omega_{n-2}.
\]
Let $\omega$ be the unique positive orbifold Hermitian form satisfying
$\omega^{n-1}=\omega_0\wedge\Theta$, and put $V_\omega=\int_X\omega^n$.
Then every Higgs orbi-bundle $(E,\theta)$ of rank $r$ satisfies
\begin{equation}\label{equa-orbi-BG}
 \Delta^\orb(E)\cdot[\omega_1]\cdots[\omega_{n-2}]
 \geq-\frac{n}{V_\omega^2}
 \left(\int_X\frac{\omega^n}{\tr_\omega\omega_0}\right)
 \sum_{a=1}^r\bigl(\mu_{\Omega,a}-\mu_\Omega(E)\bigr)^2,
\end{equation}
where $\Delta^\orb(E)=2c_2^\orb(E)-\frac{r-1}{r}c_1^\orb(E)^2$.
In particular, an $\Omega$-semistable Higgs orbi-bundle satisfies the
mixed Bogomolov--Gieseker inequality
\[
 \Delta^\orb(E)\cdot[\omega_1]\cdots[\omega_{n-2}]\geq0.
\]
\end{corollary}

\begin{proof}
Let $K$ be a fixed Hermitian metric on $E$.  Theorem~\ref{thm2} gives
Hermitian metrics $H_j$ such that, for every finite $p>0$,
\begin{equation}\label{mixed-mean-convergence}
 \sqrt{-1}\Lambda_\omega F_{H_j,\theta}
 \longrightarrow \frac{2\pi n}{V_\omega}
 \Phi_\Omega^{HN}(E,\theta,K)
\quad\text{in }L^p(K).
\end{equation}
The factor $2\pi n/V_\omega$ comes from expressing
Theorem~\ref{thm2} in terms of the unnormalised slopes used here.
Set
\[
 G_j=F_{H_j}+[\theta,\theta^{*H_j}],\qquad
 B_j=\sqrt{-1}\,G_j^\perp,\qquad M_j=\Lambda_\omega B_j.
\]
Unlike the full Hitchin--Simpson curvature, $B_j$ is of type $(1,1)$.
Writing $P_j=(\partial_{H_j}\theta)^\perp$, the Chern--Weil formula and
the type decomposition give
\begin{align}
 4\pi^2\Delta^\orb(E)\cdot[\omega_1]\cdots[\omega_{n-2}]
 &=\int_X\left\{-\tr(B_j\wedge B_j)
       +2\tr(P_j\wedge P_j^{*H_j})\right\}\wedge\Theta \notag\\
 &\geq\int_X-\tr(B_j\wedge B_j)\wedge\Theta .
 \label{mixed-CW}
\end{align}
The last inequality follows from the mixed Hodge--Riemann bilinear
relations; in an $H_j$-unitary frame the omitted term is a sum of
nonnegative $(2,2)$-forms wedged with $\Theta$.

Put $s=\tr_\omega\omega_0>0$ and
\[
 \Xi_j=B_j-\frac{M_j}{s}\,\omega_0.
\]
Since $\omega^{n-1}=\omega_0\wedge\Theta$, one has
\[
 B_j\wedge\omega^{n-1}=\frac{M_j}{n}\omega^n,
 \qquad
 \omega_0\wedge\omega^{n-1}=\frac{s}{n}\omega^n.
\]
Consequently $\Xi_j\wedge\omega_0\wedge\Theta=0$, so $\Xi_j$ is
mixed primitive.  Moreover $\omega_0^2\wedge\Theta=(s/n)\omega^n$.
The cross term therefore vanishes, and the mixed Hodge--Riemann
relations yield
\begin{align}
 -\tr(B_j\wedge B_j)\wedge\Theta
 &=-\tr(\Xi_j\wedge\Xi_j)\wedge\Theta
   -\frac{\tr(M_j^2)}{ns}\omega^n \notag\\
 &\geq-\frac{\tr(M_j^2)}{n\tr_\omega\omega_0}\omega^n.
 \label{mixed-primitive}
\end{align}
Combining \eqref{mixed-CW} and \eqref{mixed-primitive} gives
\[
 4\pi^2\Delta^\orb(E)\cdot[\omega_1]\cdots[\omega_{n-2}]
 \geq-\frac1n\int_X
 \frac{\tr(M_j^2)}{\tr_\omega\omega_0}\omega^n.
\]

Taking trace-free parts in \eqref{mixed-mean-convergence}, we obtain
\[
 M_j\longrightarrow\frac{2\pi n}{V_\omega}
 \bigl(\Phi_\Omega^{HN}-\mu_\Omega(E)\Id_E\bigr)
 \quad\text{in }L^2(K).
\]
Thus $\tr(M_j^2)$ converges in $L^1$.  Since
$(\tr_\omega\omega_0)^{-1}$ is bounded and the eigenvalues of the HN
projection are $\mu_{\Omega,1},\ldots,\mu_{\Omega,r}$, letting
$j\to\infty$ and dividing by $4\pi^2$ proves \eqref{equa-orbi-BG}.
\end{proof}

\begin{remark}\label{mixed-normalisation-check}
If $\omega_0=\cdots=\omega_{n-2}=\omega$, then
$\tr_\omega\omega_0=n$, and \eqref{equa-orbi-BG} becomes
\[
 \Delta^\orb(E)\cdot[\omega]^{n-2}
 \geq-\frac{1}{\int_X\omega^n}
 \sum_{a=1}^r\bigl(\mu_{\omega,a}-\mu_\omega(E)\bigr)^2.
\]
This specialization shows that the reciprocal trace in
\eqref{equa-orbi-BG} cannot be replaced by $\tr_\omega\omega_0$.
\end{remark}

\subsection{Proof of Theorem \ref{thm1} and Theorem \ref{thm2}} Let $(X, \omega)$ be a compact Gauduchon orbifold of complex dimension $n$ and $(E, \bar{\partial}_E , \theta )$ a  Higgs orbi-bundle of $\rank$ $r$ over $X$. After a conformal change of a Hermitian metric $K$ on $E$, we may assume
$\tr (\sqrt{-1}\Lambda_{\omega } F_{K, \theta }-\lambda  \textmd{Id}_E)=0
$
with $\lambda =\frac{2\pi}{\Vol(X,\w)}\mu_{\omega}(E)$.
For any Hermitian metric $H$ on $E$, set $h=K^{-1}H$. Then
\begin{equation}\label{id1}
	\begin{split}
		\partial _{H}-\partial_{K} =h^{-1}\partial_{K}h,\ \ F_{H}-F_{K}=\bar{\partial }_{E} (h^{-1}\partial_{K} h) .
	\end{split}
\end{equation}
As a consequence, the perturbed Hermitian-Einstein equation (\ref{eq}) can be rewritten as
\begin{equation} \label{eq2}
	\sqrt{-1}\Lambda_{\omega }\{\bar{\partial }_{E} (h^{-1}\partial_{K} h)+[\theta , h^{-1}\theta^{\ast K}h]+F_{K} \}-\lambda \cdot \textmd{Id}_E+\varepsilon \log h=0.
\end{equation}
Set
$ \Phi(H, \theta)=\sqrt{-1}\Lambda_{\omega } (F_{H}+[\theta,\theta^{* H}])-\lambda \cdot \textmd{Id}_E.$
By definition,
\begin{equation}\label{def}
	\textmd{tr}\{(\Phi(H,\theta)-\Phi(K,\theta))s\}=\langle \sqrt{-1}\Lambda_{\omega } (\bar{\partial}(h^{-1}\partial_{K}h)+[\theta,\theta^{* H}-\theta^{* {K}}]),s\rangle_{K}
\end{equation}
and
\begin{equation}\label{theta2}
	\textmd{tr}(\sqrt{-1}\Lambda_{\omega }[\theta,\theta^{* H}-\theta^{* {K}}]s)\frac{\omega^n}{n!}
	=\textmd{tr}(\sqrt{-1}h^{-1}[\theta^{* {K}},h]\wedge [\theta,s])\wedge \frac{\omega^{n-1}}{(n-1)!},
\end{equation}
where $s=\log h$. By applying Stokes's formula (see e.g. \cite[Lemma 2.7]{F19}), $\textmd{tr}(h^{-1}(\partial_{K}h)s)=\textmd{tr}(s\partial_{K}s)$ and $ \partial\bar{\partial}\omega^{n-1}=0$,  we deduce
\begin{equation} \label{theta11}
	\begin{split}
		&\int_{X}\langle \sqrt{-1}\Lambda_{\omega } (\bar{\partial}(h^{-1}\partial_{K}h)),s\rangle_{K}\frac{\omega^n}{n!}\\
		=&\int_X\sqrt{-1}\textmd{tr}(s\partial_{K}s)\wedge\frac{\bp\omega^{n-1}}{(n-1)!}
		+\int_X\sqrt{-1}\textmd{tr}(h^{-1}\partial_{K}h\wedge \bar{\partial}s)\wedge\frac{\omega^{n-1}}{(n-1)!}\\
		=&\int_X\sqrt{-1}\partial\textmd{tr}(\frac{1}{2}s^{2})\wedge\frac{\bp\omega^{n-1}}{(n-1)!}
		+\int_X\sqrt{-1}\textmd{tr}(h^{-1}\partial_{K}h\wedge \bar{\partial}s)\wedge\frac{\omega^{n-1}}{(n-1)!}\\
		=&\int_X\sqrt{-1}\textmd{tr}(h^{-1}\partial_{K}h\wedge \bar{\partial}s)\wedge\frac{\omega^{n-1}}{(n-1)!}.\\
	\end{split}
\end{equation}
From (\ref{def}), (\ref{theta2}) and (\ref{theta11}), we have
\begin{equation}\label{theta101}
	\int_{X} \textmd{tr}\{(\Phi(H,\theta)-\Phi(K,\theta))s\} \frac{\omega^n}{n!}=\int_{X} \textmd{tr} \sqrt{-1}\Lambda_{\omega }(h^{-1}D^{1,0}_{K,\theta}h\wedge \bar{\partial}_{\theta}s) \frac{\omega^n}{n!}.
\end{equation}
In \cite[p.635]{NZ}, it was proved that
\begin{equation}\label{theta21}\textmd{tr} \sqrt{-1}\Lambda_{\omega }(h^{-1}D^{1,0}_{K,\theta}h\wedge \bar{\partial}_{\theta}s)=\langle \Psi(s)(\bar{\partial}_{\theta}s),\bar{\partial}_{\theta}s\rangle_{K},\end{equation}
where
\begin{equation}\label{eq3301}
	\Psi(x,y)=
	\begin{cases}
		&\frac{e^{y-x}-1}{y-x},\ \ \ x\neq y;\\
		&\ \ \ \  1,\ \ \ \ \ \  x=y.
	\end{cases}
\end{equation}
The following proposition follows from (\ref{theta101}) and (\ref{theta21}).

\begin{proposition}\label{key}
	Let $(E, \bar{\partial}_{E}, \theta )$ be a Higgs
	orbi-bundle with a fixed Hermitian metric $K $ over a compact Gauduchon orbifold $(X, \omega )$ of complex dimension $n$. Assume $H$ is a Hermitian metric on $E$ and $s:=\log(K^{-1}H)$. Then we have
	\begin{equation}\label{eq33}
		\int_X \mathrm{tr}(\Phi(K,\theta)s)\frac{\omega^n}{n!}+\int_{X}\langle \Psi(s)(\bar{\partial}_{\theta}s),\bar{\partial}_{\theta}s\rangle_{K}\frac{\omega^n}{n!}=\int_X \mathrm{tr}(\Phi(H,\theta)s)\frac{\omega^n}{n!},
	\end{equation}
	where $\Psi$ is defined in (\ref{eq3301}).

\end{proposition}

The argument of Teleman and L\"ubke in \cite{LT} (see also Lemma 2.1 and Lemma 2.2 in \cite{NZ}) gives the following lemma.

\begin{lemma}\label{lm1}
	There exists a solution $H_{\varepsilon }$ to the perturbed Hermitian-Einstein equation (\ref{eq}) for all $\varepsilon >0$, with the following estimates:
	\begin{enumerate}
		\item $-\frac{\sqrt{-1}}{2}\Lambda_{\omega }\partial \bar{\partial }\left(| \log{h_{\varepsilon }}|_{K}^2\right)+\vps | \log{h_{\varepsilon }} |_{K}^2\leq  | \Phi(K, \theta) |_{K} | \log{h_{\varepsilon }} |_{K};$
		\item  $\max_M | \log{h_{\varepsilon }} |_{K}\leq \frac{1}{\vps}\cdot \max_M  | \Phi(K, \theta) |_{K} $;
		\item $\max_M | \log{h_{\varepsilon }} |_{K}\leq C\cdot (\| \log{h_{\varepsilon }} \|_{L^2}+\max_M  | \Phi(K, \theta) |_{K})$,
	\end{enumerate}
	where $h_\epsilon=K^{-1}H_\epsilon$ and $C$ is a constant depending only on $(M, \omega )$.
	Moreover, $\tr \Phi(K, \theta)=0$ implies
$
		\tr \log (h_{\varepsilon })=0
$  and $\tr \Phi(H_{\varepsilon }, \theta) =0$.
\end{lemma}

The Chern-Weil formula \eqref{chern-weil} for the metric $K$ gives the degree of any saturated $\theta$-invariant orbi-subsheaf $\mathcal{F}$ of $(E , \bar{\partial }_{E}, \theta )$:
\begin{equation} \label{cw}
	2\pi \textmd{deg}_{\omega}(\mathcal{F})=\int_X (\sqrt{-1}\textmd{tr}(\pi_{\mathcal{F}}^{K} \Lambda_{\omega} F_{K,\theta})-|\bar{\partial }_{\theta}\pi |_{K}^{2})\frac{\omega^{n}}{n!},
\end{equation}
where $\pi_{\mathcal{F}}^{K} $ denotes the projection onto $\mathcal{F}$ with respect to $K$.
The identity (\ref{eq33}) and Simpson's argument \cite{Simpson88} give the following proposition.

\begin{proposition}\label{lm01}
	Let $H_{\varepsilon}$ be the solution of the perturbed equation (\ref{eq}) with $\tr \Phi(K, \theta)=0$, and set $h_{\varepsilon}=K^{-1}H_{\varepsilon}$, $s_{\varepsilon}= \log h_{\varepsilon}$, $l_{\varepsilon}= \| s_{\varepsilon}\|_{L^2}$, $u_{\varepsilon}= \frac{s_{\varepsilon}}{l_{\varepsilon}}$. Assume that there is a sequence $\varepsilon_i\to 0$ such that
	\begin{equation}
		\lim\limits_{i\rightarrow \infty}\| \log(K^{-1}H_{\varepsilon_i})\|_{L^2(K)}=+\infty .
	\end{equation}
	Then
	\begin{enumerate}
		\item after passing to a subsequence, $u_{\varepsilon_{i}} \rightharpoonup u_\infty$ weakly in $L_1^2$ with $\tr u_{\infty}=0$ and $\|u_\infty\|_{L^2}=1$. The eigenvalues of $u_{\infty}$ are almost everywhere constant and not all equal.
		
		\item Assume $\mu_1< \mu_2< \cdots< \mu_l$ ($l\geq 2$) are the distinct eigenvalues of $u_\infty$. Define smooth functions $P_{\alpha}: \mathbb{R}\to \mathbb{R}$ by
		\begin{equation}
			P_{\alpha}(x)=
			\begin{cases}
				1, & x\leq \mu_{\alpha},\\
				0, & x\geq \mu_{{\alpha}+1}.
			\end{cases}
		\end{equation}
		and set $\pi_{\alpha}=P_{\alpha}(u_\infty)$ for every $1\leq \alpha \leq l-1$. Then every $\pi_{\alpha }$ determines a saturated $\theta$-invariant orbi-subsheaf $\tilde{\mathcal{E}}_{\alpha}$ of $(E, \bar{\partial}_E , \theta )$.
		
		\item Set $
			\nu :=2\pi (\sum_{\alpha=1}^{l-1}(\mu_{\alpha+1}-\mu_{\alpha})\rank(\tilde{\mathcal{E}}_{\alpha})(\mu_{\omega }(E)
			-\mu_{\omega }(\tilde{\mathcal{E}}_{\alpha})).
$
		Then
		\begin{equation}\label{nu}
			\nu \leq - \varliminf\limits_{i \rightarrow  \infty}\| \varepsilon_{i}\log(K^{-1}H_{\varepsilon_{i}})\|_{L^2}.
		\end{equation}
		
	\end{enumerate}
\end{proposition}

\begin{proof}
	Denote $H_{\varepsilon_i}$ by $H_{i}$ and set $h_i=K^{-1}H_{i}$, $s_i= \log h_i$, $l_i= \| s_i\|_{L^2}$, $u_i= \frac{s_i}{l_i}$ for simplicity. By Lemma \ref{lm1}, we get
	\begin{equation}
		\tr u_i=0 ,  \quad \|u_i\|_{L^2}=1 ,\ \  \|u_i\|_{L^\infty }\leq \hat{C} \quad \text{and} \quad \varepsilon_i l_i \leq \bar{C}.
	\end{equation}
	Equation (\ref{eq33}) gives
	\begin{equation}\label{key0}
		\int_X (\tr(\Phi(K,\theta)u_i)+l_i \langle\Psi (l_i u_i)(\bar{\partial}_\theta u_i), \bar{\partial}_\theta u_i\rangle_K)\frac{\omega^n}{n!}=-\varepsilon_i l_i.
	\end{equation}
	Using (\ref{key0}) and Simpson's argument in \cite[Lemma 5.4]{Simpson88}, we obtain
$
		\|D_{K,\theta} u_i\|_{L^2}< \tilde{C}.
$
	Thus $u_i$ are uniformly bounded in $L^\infty$ and $L_1^2$. After passing to a subsequence, still denoted by $\{u_i\}$, we have $u_i \rightharpoonup u_\infty$ weakly in $L_1^2$. By the Kondrachov compactness theorem (\cite[Theorem 7.22]{GT}), $L_1^2$ embeds compactly in $L^q$ for every $0<q< \frac{2n}{n-1}$. Hence
	\begin{equation}\label{u4}
		\lim_{i\rightarrow \infty} \|u_i-u_\infty\|_{L^q}=0
	\end{equation}
	and
	\begin{equation}\label{lqcvg}
		\lim_{i\rightarrow \infty} \|\sqrt{-1}\Lambda_{\omega}F_{H_{i}, \theta }-\lambda\Id_E+\delta u_\infty\|_{L^q}=0
	\end{equation}
	for any $0<q< \frac{2n}{n-1}$. Hence $\|u_\infty\|_{L^2}=1$.
	
	By (\ref{key0}) and Simpson's argument (\cite[Lemma 5.5]{Simpson88}), the eigenvalues of $u_{\infty}$ are almost everywhere constant and not all equal. For $\alpha < l$, the argument in \cite[p.~887]{Simpson88} gives
	$$\pi_{\alpha}\in L^2_1;\  \pi^2_{\alpha}=\pi_{\alpha}=\pi_{\alpha}^{*K};\ (\mathrm{Id}-\pi_{\alpha})\bp \pi_{\alpha}=0;\ (\mathrm{Id}-\pi_{\alpha})[\theta, \pi_{\alpha}]=0.$$
	By the orbifold version of Uhlenbeck-Yau regularity \cite{UY86} for $L_1^2$-orbi-subbundles, as stated in {\cite[Lemma 28]{F20}}, $\pi_{\alpha }$ determines a saturated orbi-subsheaf $\tilde{\mathcal{E}}_{\alpha}$ of $E$. The preceding identities and Lemma \ref{lem-saturation-invariant} imply that $\tilde{\mathcal{E}}_{\alpha}$ is $\theta$-invariant. Using (\ref{key0}) and the argument of \cite[Lemma 5.4]{Simpson88} (see also \cite[(3.23)]{NZ}), we obtain
	\begin{equation}\label{key00}
		%\begin{equation} \label{seq4}
		\varliminf\limits_{i \rightarrow  \infty}\| \varepsilon_{i}\log(K^{-1}H_{\varepsilon_{i}})\|_{L^2}+\int_X \textmd{tr}(\Phi(K,\theta)u_{\infty})\frac{\omega^n}{n!}+\int_{X}\langle \varsigma(u_{\infty})(\bar{\partial}_{\theta}u_{\infty}),\bar{\partial}_{\theta}u_{\infty}\rangle_{K}\frac{\omega^n}{n!}\leq 0,
	\end{equation}
	where $ \zeta \in C^{\infty} (\mathbb R\times \mathbb R, \mathbb R^+)$ satisfies $\zeta(x,y)<(x-y)^{-1}$ whenever $x>y$. 
	
	Set $\pi_l=\Id_E$ and $\pi_0=0$. Then $\label{u1}u_{\infty}=\mu_l \cdot \textmd{Id}_E-\sum\limits_{\alpha=1}^{l-1}(\mu_{\alpha+1}-\mu_{\alpha})\pi_{\alpha}.$ From $\textmd{tr}(u_{\infty})=0$, it holds
	\begin{equation} \label{seq5}
		\mu_{l}\cdot\rank(E)=\sum_{\alpha=1}^{l-1}(\mu_{\alpha+1}-\mu_{\alpha})\cdot\rank(\tilde{\mathcal{E}}_{\alpha}),
	\end{equation}
	and then
	\begin{equation}\label{key1000}
		\begin{split}
			&\nu =2\pi\big(\mu_l\cdot \textmd{deg}_{\omega}(E)-\sum_{\alpha=1}^{l-1}(\mu_{\alpha+1}-\mu_{\alpha})\cdot\textmd{deg}_{\omega}(\tilde{\mathcal{E}}_{\alpha})\big)\\
			&=\mu_l\int_X \sqrt{-1}\textmd{tr}(\Lambda_{\omega } F_{K,\theta})
			\frac{\w^n}{n!}-\sum_{\alpha=1}^{l-1}(\mu_{\alpha+1}-\mu_{\alpha})\int_X (\sqrt{-1}\textmd{tr}(\pi_{\alpha}\Lambda_{\omega } F_{K,\theta})-|\bar{\partial}_{\theta}\pi_{\alpha}|^2)\frac{\w^n}{n!}\\
			&=\int_X \textmd{tr}\{(\mu_l\cdot \textmd{Id}_E-\sum_{\alpha=1}^{l-1}(\mu_{\alpha+1}-\mu_{\alpha})\pi_{\alpha})(\sqrt{-1}\Lambda_{\omega } F_{K,\theta})\}\frac{\w^n}{n!}+\sum_{\alpha=1}^{l-1}(\mu_{\alpha+1}-\mu_{\alpha})\int_X|\bar{\partial}_{\theta}\pi_{\alpha}|^2\frac{\w^n}{n!}\\
			&=\int_X \textmd{tr}(u_{\infty}\sqrt{-1}\Lambda_{\omega } F_{K,\theta})\frac{\w^n}{n!}
			+\int_X\langle \sum_{\alpha=1}^{l-1}(\mu_{\alpha+1}-\mu_{\alpha})(\textmd{d} P_{\alpha})^2(u_{\infty})(\bar{\partial}_{\theta}u_{\infty}),\bar{\partial}_{\theta}u_{\infty}\rangle_{K}\frac{\w^n}{n!}
		\end{split}
	\end{equation}
	where the function $\textmd{d} P_{\alpha}: \mathbb{R}\times\mathbb{R}\rightarrow \mathbb{R}$ is defined by
	$$
		\textmd{d} P_{\alpha}(x,y)=
		\begin{cases}
			\dfrac{P_{\alpha}(x)-P_{\alpha}(y)}{x-y}, & x\neq y,\\
			P'_{\alpha}(x), &x=y.
		\end{cases}
	$$
	A direct computation gives
	$
	\sum_{\alpha=1}^{l-1}(\mu_{\alpha+1}-\mu_{\alpha})(\textmd{d} P_{\alpha})^2(\mu_{\beta},\mu_{\gamma})=|\mu_{\beta}-\mu_{\gamma}|^{-1},
	$
	if $\mu_{\beta}\neq \mu_{\gamma}$. Hence, by (\ref{key00}) and applying the arguments in \cite[p.~793-794]{LZ}, we obtain
	\begin{equation}\label{eq320}
		\begin{split}
			\nu = &\int_X \textmd{tr}(u_{\infty}\sqrt{-1}\Lambda_{\omega } F_{K,\theta})\frac{\w^n}{n!}
			+\int_X\langle \sum_{\alpha=1}^{l-1}(\mu_{\alpha+1}-\mu_{\alpha})(\textmd{d} P_{\alpha})^2(u_{\infty})(\bar{\partial}_{\theta}u_{\infty}),\bar{\partial}_{\theta}u_{\infty}\rangle_{K}\frac{\w^n}{n!},\\
			\leq & -\varliminf\limits_{i \rightarrow  \infty}\| \varepsilon_{i}\log(K^{-1}H_{\varepsilon_{i}})\|_{L^2}.
		\end{split}
	\end{equation}

\end{proof}

\begin{lemma}\label{deg01}
	Under the assumptions of Proposition \ref{lm01}, suppose
	\begin{equation}
		\lim\limits_{i\rightarrow \infty}\| \varepsilon_i\log(K^{-1}H_{\varepsilon_i})\|_{L^2, K}=\delta \geq 0.
	\end{equation} Let $\mathcal{F}\subset E$ be a saturated $\theta $-invariant orbi-subsheaf, then
	\begin{equation}\label{lambda1}
		\frac{2\pi \deg_{\omega}(\mathcal{F})}{\rank(\mathcal{F})}\leq (\lambda -\delta \mu_{1})\Vol_{\omega}(M),
	\end{equation}
	\begin{equation}\label{lambda1a}
		\frac{2\pi \deg_{\omega}(E/\mathcal{F})}{\rank(E/\mathcal{F})}\geq (\lambda -\delta \mu_{l})\Vol_{\omega}(M),
	\end{equation}
	\begin{equation}\label{subdeg}
		2\pi \deg_{\omega }(\tilde{\mathcal{E}}_{\alpha})\leq \Vol_{\omega }(M)\sum_{\beta =1}^{\alpha }(\lambda -\delta \mu_{\beta })\big(\rank({\tilde{\mathcal{E}}_{\beta }})-\rank({\tilde{\mathcal{E}}_{\beta -1}})\big),
	\end{equation}
	\begin{equation}\label{quodeg}
		2\pi \deg_{\omega }(E/\tilde{\mathcal{E}}_{\alpha})\geq \Vol_{\omega }(M)\sum_{\beta =\alpha +1}^{l }(\lambda -\delta \mu_{\beta })\big(\rank({\tilde{\mathcal{E}}_{\beta }})-\rank({\tilde{\mathcal{E}}_{\beta -1}})\big) .
	\end{equation}
\end{lemma}

\begin{proof}
	By (\ref{cw}),
	\begin{equation}\label{deg1}
		\begin{split}
			2\pi\deg_{\omega}(\mathcal{F})=&\int_X(\tr(\pi^{H_i}_{\mathcal{F}}\sqrt{-1}\Lambda_{\omega}F_{H_{i}, \theta })-|\bar{\partial}_{\theta }\pi^{H_i}_{\mathcal{F}}|_{H_{i}}^2)\frac{\omega^n}{n!}\\
			\leq&\int_X\tr(h_i^{\frac{1}{2}}\pi^{H_i}_{\mathcal{F}}h_i^{-\frac{1}{2}}h_i^{\frac{1}{2}}(\sqrt{-1}\Lambda_{\omega}F_{H_{i},\theta})h_i^{-\frac{1}{2}})\frac{\omega^n}{n!}\\
			=&\int_X\tr(h_i^{\frac{1}{2}}\pi^{H_i}_{\mathcal{F}}h_i^{-\frac{1}{2}}(\lambda\Id_E-\varepsilon_i\log h_i))\frac{\omega^n}{n!}\\
			=&\lambda\cdot\rank(\mathcal{F})\cdot\Vol_{\omega}(X)+\int_X\tr(h_i^{\frac{1}{2}}\pi^{H_i}_{\mathcal{F}}h_i^{-\frac{1}{2}}\varepsilon_i l_i(u_\infty-u_i))\frac{\omega^n}{n!}\\
			-&\int_X\tr(h_i^{\frac{1}{2}}\pi^{H_i}_{\mathcal{F}}h_i^{-\frac{1}{2}}\varepsilon_i l_i u_\infty)\frac{\omega^n}{n!}
		\end{split}
	\end{equation}
	and
	\begin{equation}\label{deg1a}
		\begin{split}
			2\pi\deg_{\omega}(E/\mathcal{F})=&\int_X(\tr((\Id_{E}-\pi^{H_i}_{\mathcal{F}})\sqrt{-1}\Lambda_{\omega}F_{H_{i}, \theta })+|\bar{\partial}_{\theta }\pi^{H_i}_{\mathcal{F}}|_{H_{i}}^2)\frac{\omega^n}{n!}\\
			\geq&\int_X\tr(h_i^{\frac{1}{2}}(\Id_{E}-\pi^{H_i}_{\mathcal{F}})h_i^{-\frac{1}{2}}(\lambda\Id_E-\varepsilon_i\log h_i))\frac{\omega^n}{n!}\\
			=&\int_X\tr(h_i^{\frac{1}{2}}(\Id_{E}-\pi^{H_i}_{\mathcal{F}})h_i^{-\frac{1}{2}}\varepsilon_i l_i(u_\infty-u_i))\frac{\omega^n}{n!}\\
			+&\int_X\tr(h_i^{\frac{1}{2}}(\Id_{E}-\pi^{H_i}_{\mathcal{F}})h_i^{-\frac{1}{2}}(\lambda\Id_{E}-\varepsilon_i l_i u_\infty))\frac{\omega^n}{n!}
		\end{split}
	\end{equation}
	for any $i$. Note that $\mu_{1}$ (resp. $\mu_{l}$) is the smallest (resp. largest) eigenvalue of $u_{\infty}$. Thus
	\begin{equation}\label{u5}
		-\tr(h_i^{\frac{1}{2}}\pi^{H_i}_{\mathcal{F}}h_i^{-\frac{1}{2}}u_\infty)\leq -\mu_{1}\rank(\mathcal{F})
	\end{equation}
	and
	\begin{equation}\label{u5a}
		-\tr(h_i^{\frac{1}{2}}(\Id_E- \pi^{H_i}_{\mathcal{F}})h_i^{-\frac{1}{2}}u_\infty)\geq -\mu_{l}(\rank(E)-\rank(\mathcal{F})).
	\end{equation}
	The argument of \cite[formula (2.27)]{LZZ21} gives
	\begin{equation}\label{u6}
		-\tr(h_i^{\frac{1}{2}}\pi^{H_i}_{\tilde{\mathcal{E}}_{\alpha}}h_i^{-\frac{1}{2}}u_\infty)\leq \sum_{\beta =1}^{\alpha }(-\mu_{\beta })\big(\rank({\tilde{\mathcal{E}}_{\beta }})-\rank({\tilde{\mathcal{E}}_{\beta -1}})\big)
	\end{equation}
	for any $1\leq \alpha \leq l$.
	Combining (\ref{u4}), (\ref{deg1}), (\ref{deg1a}), (\ref{u5}), (\ref{u5a}) and (\ref{u6}), we obtain (\ref{lambda1}), (\ref{lambda1a}) and (\ref{subdeg}).
	By (\ref{seq5}),
	\begin{equation}\label{u3} \sum_{\beta=1}^l \mu_{\beta}\big(\rank({\tilde{\mathcal{E}}_{\beta }})-\rank({\tilde{\mathcal{E}}_{\beta -1}})\big) =0.\end{equation}
	Then (\ref{subdeg}) and (\ref{u3}) imply (\ref{quodeg}).
	
\end{proof}

Recall that $\lambda_{mU}(H_i, \theta , \omega)$  is the average of the largest eigenvalue function $\lambda_U(H_i, \theta , \omega)$ of $\sqrt{-1}\Lambda_{\omega}F_{H_{i}, \theta }$, while $\lambda_{mL}(H_i, \theta , \omega)$ represents the average of the smallest eigenvalue function $\lambda_L(H_i, \theta , \omega)$ of $\sqrt{-1}\Lambda_{\omega}F_{H_{i}, \theta }$. The Chern-Weil formula (\ref{cw}) gives
\begin{equation}\label{lambda2}
	\varliminf_{i\to \infty}\lambda_{mU}(H_i, \theta , \omega)\Vol_{\omega}(X)\geq \sup_{\mathcal{F}} 2\pi(\frac{\deg_{\omega}(\mathcal{F})}{\rank(\mathcal{F})})
\end{equation}
and
\begin{equation}\label{lambda3}
	\varlimsup_{i\to \infty}\lambda_{mL}(H_i, \theta , \omega)\Vol_{\omega}(X)\leq \inf_{\mathcal{F}} 2\pi(\frac{\deg_{\omega}(\mathcal{E/\mathcal{F}})}{\rank(\mathcal{E/\mathcal{F}})}),
\end{equation}
where $\mathcal{F}$ runs over all the torsion-free $\theta$-invariant orbi-subsheaves of $(E,\theta)$.

Assume $e_1^{i}$ is an eigenvector corresponding to the eigenvalue $\lambda_U(H_i, \theta , \omega)$ of $\sqrt{-1}\Lambda_{\omega}F_{H_{i}, \theta }$, and $|e_1^{i}|_K=1$. Then
\begin{equation}
	\begin{split}
		\lambda_U(H_i, \theta , \omega)=&\langle \sqrt{-1}\Lambda_{\omega}F_{H_{i}, \theta }(e_1^{i}), e_1^{i}\rangle_K\\
		=& \langle (\Phi(H_{i}, \theta)+ \delta u_\infty)e_1^{i}, e_1^{i}\rangle_K
		+\langle (\lambda\Id_E- \delta u_\infty)e_1^{i}, e_1^{i}\rangle_K\\
		\leq& |\delta u_\infty -\varepsilon_{i}l_{i}u_{i}|_{K}+ \lambda-\delta u_1.
	\end{split}
\end{equation}
Together with (\ref{u4}), it follows that
\begin{equation}\label{lambda4}
	\varlimsup_{i\to \infty}\lambda_{mU}(H_i, \theta , \omega)
	=\varlimsup_{i\to \infty}\frac{1}{\Vol_{\omega }(X)}\int_X \lambda_U(H_i, \theta , \omega)\frac{\omega^n}{n!}
	\leq \lambda-\delta u_1 .
\end{equation}
Similarly,
\begin{equation}\label{lambda5}
	\varliminf_{i\to \infty}\lambda_{mL}(H_i, \theta , \omega)\geq \lambda-\delta u_l.
\end{equation}

\begin{proposition}\label{lm02}
	Under the assumptions of Lemma \ref{deg01},
	if $\delta >0$ , then
	\begin{equation*}
		0= \tilde{\mathcal{E}}_{0}\subset \tilde{\mathcal{E}}_{1} \subset \tilde{\mathcal{E}}_{2} \subset\cdots \subset \tilde{\mathcal{E}}_{l}=E
	\end{equation*}
	is the Harder-Narasimhan filtration of the Higgs orbi-bundle $(E, \bar{\partial}_E , \theta )$.
\end{proposition}

\begin{proof}
	Set
	$
		r_l=r=\rank{E}, r_{\alpha }=\rank{\tilde{\mathcal{E}}_{\alpha}}=\tr \pi_{\alpha}, \lambda_{\alpha}= \lambda-\delta\mu_{\alpha}
$
	for all $1\leq \alpha \leq l$,
	and \begin{equation}\label{infty2}\tilde{u}_\infty:=\lambda\Id_E-\delta u_\infty= \sum_{\alpha=1}^{l}\lambda_{\alpha}(\pi_{\alpha}-\pi_{\alpha -1}).
	\end{equation}
	As $\|u_\infty\|_{L^2}=1$, this implies that
$
		\sum_{\alpha=1}^l\mu_{\alpha}^2(r_{\alpha}-r_{\alpha -1})\Vol_{\omega}(M)=1,
$
	and then
	\begin{equation}\label{u2}
		\sum_{\alpha=1}^l(\lambda-\lambda_{\alpha})^2(r_{\alpha}-r_{\alpha -1})=\frac{\delta^2}{\Vol_{\omega}(M)}.
	\end{equation}
	In view of (\ref{nu}), (\ref{u2}) and (\ref{u3}), we obtain
	\begin{equation}\label{key1}
		\begin{split}
			0\geq &\delta^2+\delta\nu
			=\delta^2+2\pi\sum_{\alpha =1}^l(\lambda-\lambda_{\alpha})(\deg_{\omega}(\tilde{\mathcal{E}}_{\alpha})-\deg_{\omega}(\tilde{\mathcal{E}}_{\alpha -1}))\\
			=&\sum_{\alpha =1}^l(\lambda-\lambda_{\alpha})(2\pi(\deg_{\omega}(\tilde{\mathcal{E}}_{\alpha})-\deg_{\omega}(\tilde{\mathcal{E}}_{\alpha -1}))-\lambda_{\alpha}(r_{\alpha}-r_{\alpha-1})\Vol_{\omega}(X)).
		\end{split}
	\end{equation}
	By $\eqref{u3}$,
	\begin{equation}\label{degeq}
		2\pi\deg_{\omega}(E)=\Vol_{\omega}(X)\cdot\sum_{\alpha =1}^l\lambda_{\alpha }(r_{\alpha }-r_{\alpha -1}).
	\end{equation}
	By (\ref{subdeg}) and (\ref{degeq}),
	\begin{equation}\label{key2}
		\begin{aligned}
			&\sum_{\alpha =1}^l(\lambda-\lambda_{\alpha})(2\pi(\deg_{\omega}(\tilde{\mathcal{E}}_{\alpha})-\deg_{\omega}(\tilde{\mathcal{E}}_{\alpha -1}))-\lambda_{\alpha}(r_{\alpha}-r_{\alpha-1})\Vol_{\omega}(X))\\
			=&\sum_{\alpha =1}^{l}(\lambda-\lambda_{\alpha })\bigg(2\pi\deg_{\omega}(\tilde{\mathcal{E}}_{\alpha})-\Vol_{\omega }(M)\cdot\sum_{\beta =1}^{\alpha }\lambda_{\beta }(r_{\beta }-r_{\beta -1})\\
			&-\bigg(2\pi\deg_{\omega}(\tilde{\mathcal{E}}_{\alpha -1})-\Vol_{\omega}(X)\cdot\sum_{\beta =1}^{\alpha -1}\lambda_{\beta }(r_{\beta }-r_{\beta -1})\bigg)\bigg)\\
			=&\sum_{\alpha =1}^{l-1}(\lambda_{\alpha +1}-\lambda_{\alpha })\bigg(2\pi\deg_{\omega}(\tilde{\mathcal{E}}_{\alpha})-\Vol_{\omega }(X)\cdot\sum_{\beta =1}^{\alpha }\lambda_{\beta }(r_{\beta }-r_{\beta -1})\bigg)\\
			&+(\lambda-\lambda_l)\bigg(2\pi\deg_{\omega }(E)-\Vol_{\omega }(X)\cdot\sum_{\beta =1}^{l }\lambda_{\beta }(r_{\beta }-r_{\beta -1})\bigg)\\
			=&\sum_{\alpha =1}^{l-1}(\lambda_{\alpha +1}-\lambda_{\alpha })\bigg(2\pi\deg_{\omega}(\tilde{\mathcal{E}}_{\alpha})-\Vol_{\omega }(X)\cdot\sum_{\beta =1}^{\alpha }\lambda_{\beta }(r_{\beta }-r_{\beta -1})\bigg)\\
			\geq &0.
		\end{aligned}
	\end{equation}
	Since $\lambda_{\alpha+1}<\lambda_\alpha$, combining (\ref{subdeg}), (\ref{key1}), (\ref{key2}) and (\ref{degeq}) gives
	\begin{equation}
		2\pi\deg_{\omega}(\tilde{\mathcal{E}}_{\alpha})=\Vol_{\omega }(X)\cdot\sum_{\beta =1}^{\alpha }\lambda_{\beta }(r_{\beta }-r_{\beta -1})
	\end{equation}
	and then
	\begin{equation}\label{lambda06}
		\frac{2\pi(\deg_{\omega}(\tilde{\mathcal{E}}_{\alpha})-\deg_{\omega}(\tilde{\mathcal{E}}_{\alpha -1}))}{r_{\alpha }-r_{\alpha -1}}= \Vol_{\omega }(X)\cdot\lambda_{\alpha }
	\end{equation}
	for $1\leq \alpha \leq l$. 
	
	Combining (\ref{lambda2}), (\ref{lambda3}), (\ref{lambda4}) and (\ref{lambda5}), we obtain
	\begin{equation}
		\begin{split}
			\sup_{\mathcal{F}} 2\pi(\frac{\deg_{\omega }(\mathcal{F})}{\rank(\mathcal{F})})\leq &\varliminf_{i\to \infty}\lambda_{mU}(H_i, \theta , \omega)\Vol_{\omega }(X)
			\leq \lambda_1 \Vol_{\omega }(X)\\
			=& 2\pi\frac{\deg_{\omega}(\tilde{\mathcal{E}}_{1})}{\rank(\tilde{\mathcal{E}}_{1})}
			\leq \sup_{\mathcal{F}} 2\pi(\frac{\deg(\mathcal{F})}{\rank(\mathcal{F})})
		\end{split}
	\end{equation}
	and
	\begin{equation}
		\begin{split}
			\inf_{\mathcal{F}} 2\pi(\frac{\deg_{\omega }(E/\mathcal{F})}{\rank(E/\mathcal{F})})\geq & \varlimsup_{i\to \infty}\lambda_{mL}(H_i, \theta , \omega)\Vol_{\omega }(X)
			\geq \lambda_l \Vol_{\omega }(X)\\
			=& 2\pi\cdot\frac{\deg_{\omega }(E)-\deg_{\omega}(\tilde{\mathcal{E}}_{l-1})}{\rank(E)-\rank(\tilde{\mathcal{E}}_{l-1})}
			\geq \inf_{\mathcal{F}} 2\pi(\frac{\deg_{\omega }(E/\mathcal{F})}{\rank(E/\mathcal{F})}),
		\end{split}
	\end{equation}
	where $\mathcal{F}$ runs over all the saturated $\theta$-invariant orbi-subsheaves of $(E,\theta)$. It follows that
	\begin{equation}\label{355}
		\begin{split}
			\lim_{i\to \infty}\lambda_{mU}(H_i, \theta , \omega)\Vol_{\omega }(X)=\lambda_1 \Vol_{\omega }(X)= 2\pi\frac{\deg_{\omega}(\tilde{\mathcal{E}}_{1})}{\rank(\tilde{\mathcal{E}}_{1})}=\sup_{\mathcal{F}} 2\pi(\frac{\deg(\mathcal{F})}{\rank(\mathcal{F})})
		\end{split}
	\end{equation}
	and
	\begin{equation}\label{356}
		\begin{split}
			\lim_{i\to \infty}\lambda_{mL}(H_i, \theta , \omega)\Vol_{\omega }(X)=&\lambda_l \Vol_{\omega }(X)
			= 2\pi\cdot\frac{\deg_{\omega }(E/\tilde{\mathcal{E}}_{l-1})}{\rank(E/\tilde{\mathcal{E}}_{l-1})}\\
			=&\inf_{\mathcal{F}} 2\pi(\frac{\deg_{\omega }(E/\mathcal{F})}{\rank(E/\mathcal{F})}).
		\end{split}
	\end{equation}

	Suppose $\mathcal{F}$ is a saturated $\theta$-invariant orbi-subsheaf of $E$ such that $\rank(\mathcal{F})> r_{\alpha -1}$ for some $\alpha \geq 2$. 
	Notice that
	\begin{equation}\label{kk01}
		\begin{split}
			2\pi\deg(\mathcal{F})=&\int_X(\tr(\pi^{H_i}_{\mathcal{F}}\sqrt{-1}\Lambda_{\omega}F_{H_{i}, \theta })-|\bar{\partial}_{\theta}\pi^{H_i}_{\mathcal{F}}|_{H_{i}}^2)\frac{\omega^n}{n!}\\
			\leq&\int_X\tr(h_i^{\frac{1}{2}}\pi^{H_i}_{\mathcal{F}}h_i^{-\frac{1}{2}}(\sqrt{-1}\Lambda_{\omega}F_{H_{i}, \theta }-\tilde{u}_\infty))\frac{\omega^n}{n!}
			+ \int_X\tr(h_i^{\frac{1}{2}}\pi^{H_i}_{\mathcal{F}}h_i^{-\frac{1}{2}}\tilde{u}_\infty)\frac{\omega^n}{n!}.
		\end{split}
	\end{equation}
	By choosing a suitable basis of $E$ at the given point, and following the argument in \cite[formula (2.65)]{LZZ21}, we have
	\begin{equation}\label{kk02}
		\tr(h_i^{\frac{1}{2}}\pi^{H_i}_{\mathcal{F}}h_i^{-\frac{1}{2}}\tilde{u}_\infty)\leq \sum_{\beta=1}^{\alpha -1}\lambda_{\beta}(r_{\beta}-r_{\beta -1})+ \lambda_{\alpha }\cdot (\rank(\mathcal{F})-r_{\alpha -1}).
	\end{equation}
	Using (\ref{u3}), (\ref{lambda06}), (\ref{kk01}) and (\ref{kk02}), we obtain
	\begin{equation}\label{4.1.1}
		\begin{split}
			2\pi\deg(\mathcal{F})\leq &(\sum_{\beta=1}^{\alpha -1}\lambda_{\beta}(r_{\beta}-r_{\beta -1})+ \lambda_{\alpha }\cdot (\rank(\mathcal{F})-r_{\alpha -1}))\Vol_{\omega }(X)\\
			=& 2\pi\sum_{\beta=1}^{\alpha -1}(\deg_{\omega}(\tilde{\mathcal{E}}_{\beta})-\deg_{\omega}(\tilde{\mathcal{E}}_{\beta -1}))+ \lambda_{\alpha }\cdot (\rank(\mathcal{F})-r_{\alpha -1})\Vol_{\omega }(M)\\
			=&  2\pi \deg_{\omega}(\tilde{\mathcal{E}}_{\alpha -1})+ \lambda_{\alpha }\cdot (\rank(\mathcal{F})-r_{\alpha -1})\Vol_{\omega }(M).\\
		\end{split}
	\end{equation}
	It follows immediately that
	\begin{equation}\label{HN11}
		\begin{split}
			&\frac{2\pi(\deg_{\omega }(\mathcal{F})-\deg_{\omega}(\tilde{\mathcal{E}}_{\alpha -1}))}{\rank(\mathcal{F})-\rank(\tilde{\mathcal{E}}_{\alpha -1})}\leq  \frac{2\pi(\deg_{\omega}(\tilde{\mathcal{E}}_{\alpha })-\deg_{\omega}(\tilde{\mathcal{E}}_{\alpha -1}))}{\rank(\tilde{\mathcal{E}}_{\alpha })-\rank(\tilde{\mathcal{E}}_{\alpha -1})}\\
			&=\lambda_{\alpha }\Vol_{\omega }(X)<  \lambda_{\alpha -1}\Vol_{\omega }(X).\\
		\end{split}
	\end{equation}
	
	We now show that
	$	0=\mathcal{E}_{0}\subset \mathcal{E}_{1}\subset \cdots \subset \mathcal{E}_{l}=E
	$
	is the Harder-Narasimhan filtration of the Higgs orbi-bundle $(E, \bar{\partial}_E , \theta )$.
	By (\ref{355}),
	\begin{equation}
		\frac{\deg_{\omega}(\tilde{\mathcal{E}}_{1})}{\rank(\tilde{\mathcal{E}}_{1})}= \sup_{\mathcal{F}\subset E} (\frac{\deg(\mathcal{F})}{\rank(\mathcal{F})}).
	\end{equation}
	If $\rank(\mathcal{F})> \rank(\tilde{\mathcal{E}}_{1})$, then (\ref{HN11}) gives
	\begin{equation}
		\deg_{\omega}(\mathcal{F})-\deg_{\omega}(\tilde{\mathcal{E}}_{1})<\frac{\rank(\mathcal{F})-\rank(\tilde{\mathcal{E}}_{1})}{\rank(\tilde{\mathcal{E}}_{1})}\deg_{\omega}(\tilde{\mathcal{E}}_{1}),
	\end{equation} 
	and then
	\begin{equation}\frac{\deg_{\omega}(\mathcal{F})}{\rank(\mathcal{F})}< \frac{\deg_{\omega}(\tilde{\mathcal{E}}_{1})}{\rank(\tilde{\mathcal{E}}_{1})},\end{equation}
	Next consider
	$
	0 \subset \tilde{\mathcal{E}}_{\alpha } \subset \hat{\mathcal{F}} \subset E,$
	where $\hat{\mathcal{F}}$ is a torsion-free $\theta$-invariant orbi-subsheaf of $(E,\theta)$, $\rank (\hat{\mathcal{F}})>\rank (\tilde{\mathcal{E}}_{\alpha })$ and $\alpha\geq 1$.
	Applying (\ref{HN11}) again gives
	\begin{equation}
		\frac{\deg_{\omega }(\hat{\mathcal{F}})-\deg_{\omega}(\tilde{\mathcal{E}}_{\alpha })}{\rank(\hat{\mathcal{F}})-\rank(\tilde{\mathcal{E}}_{\alpha })}\leq \frac{\deg_{\omega}(\tilde{\mathcal{E}}_{\alpha +1})-\deg_{\omega}(\tilde{\mathcal{E}}_{\alpha })}{r_{\alpha +1}-r_{\alpha }},
	\end{equation}
	and if $\rank (\hat{\mathcal{F}})>\rank (\tilde{\mathcal{E}}_{\alpha +1})$, then
	\begin{equation}
		\frac{\deg_{\omega }(\hat{\mathcal{F}})-\deg_{\omega}(\tilde{\mathcal{E}}_{\alpha })}{\rank(\hat{\mathcal{F}})-\rank(\tilde{\mathcal{E}}_{\alpha })}< \frac{\deg_{\omega}(\tilde{\mathcal{E}}_{\alpha +1})-\deg_{\omega}(\tilde{\mathcal{E}}_{\alpha })}{r_{\alpha +1}-r_{\alpha }}.
	\end{equation}
	This proves the proposition.
	
\end{proof}

\begin{proof}[Proof of Theorem \ref{thm1}] Let $H_{\varepsilon}$ be solutions of the perturbed Hermitian-Einstein equation (\ref{eq}) for $0<\varepsilon \leq 1$. 
	In light of (\ref{nu}) in Proposition \ref{lm01}, if the Higgs orbi-bundle $(E, \theta )$ is $\omega$-stable, then $\| \log(K^{-1}H_{\varepsilon})\|_{L^2}$ is uniformly bounded. This gives a uniform $C^{0}$-estimate. Equation (\ref{eq}) and elliptic estimates give uniform $C^{\infty}$-estimates. After passing to a subsequence, $H_{\varepsilon} \rightarrow H_{\infty}$ in the $C^{\infty}$-topology, and
	\begin{equation}
		\sqrt{-1}\Lambda_{\omega} (F_{H_{\infty}} +[\theta , \theta^{* H_{\infty}}])
		=\lambda\cdot \mathrm{Id}_{E}.
	\end{equation}
	Similarly, (\ref{nu}) shows that if the Higgs orbi-bundle $(E, \bar{\partial }_{E}, \theta )$ is $\omega$-semistable, then
	\begin{equation}
		\lim_{\varepsilon \rightarrow 0} \|\varepsilon \log(K^{-1}H_{\varepsilon})\|_{L^2(K)} =0.
	\end{equation}
	From Lemma \ref{lm1} (iii), it follows that
	\begin{equation}
			\lim_{\varepsilon \rightarrow 0}\|\sqrt{-1}\Lambda_{\omega } (F_{H_{\varepsilon}}+[\theta, \theta^{\ast H_{\varepsilon}}])-\lambda \cdot \textmd{Id}_E\|_{L^{\infty}( H_{\varepsilon})} =\lim_{\varepsilon \rightarrow 0}\|\varepsilon \log (K^{-1}H_{\varepsilon})\|_{L^{\infty}( H_{\varepsilon})}
			=0.
	\end{equation}
	This proves the existence of an approximate Hermitian-Einstein structure.
\end{proof}

\begin{proof}[Proof of Theorem \ref{thm2}] It remains to consider the case in which $(E,\bar\partial_E,\theta)$ is not $\omega$-semistable. Write $s_\varepsilon=\log(K^{-1}H_\varepsilon)$ and let $M_\varepsilon$ be the mean curvature. The Chern--Weil formula for a maximal destabilising subsheaf and the perturbed equation give
	\[
	 2\pi\mu_{1,\omega}(E)
	 \leq\int_X\lambda_{\max}(M_\varepsilon)\,dV_\omega
	 \leq\lambda\Vol_\omega(X)
	   +\varepsilon\Vol_\omega(X)^{1/2}\|s_\varepsilon\|_{L^2}.
	\]
	Since $\lambda\Vol_\omega(X)=2\pi\mu_\omega(E)$ and
	$\mu_{1,\omega}(E)>\mu_\omega(E)$, it follows that
	\[
	 \liminf_{\varepsilon\to0}\varepsilon\|s_\varepsilon\|_{L^2}
	 \geq\frac{2\pi(\mu_{1,\omega}(E)-\mu_\omega(E))}
	 {\Vol_\omega(X)^{1/2}}>0.
	\]
	Consequently one can choose a sequence of solutions $H_{\varepsilon_i}$ such that
	\begin{equation}
		\lim\limits_{\varepsilon_i\rightarrow 0}\| \log(K^{-1}H_{\varepsilon_i})\|_{L^2}=+\infty \text{ and }
		\lim\limits_{\varepsilon_i\rightarrow 0 }\| \varepsilon_i\log(K^{-1}H_{\varepsilon_i})\|_{L^2}=\delta > 0.
	\end{equation}
	Proposition \ref{lm02} gives
		\begin{align*}
			\lambda \Id_E-\delta u_\infty=\sum_{\alpha =1}^l\lambda_{\alpha}(\pi_{\alpha}-\pi_{\alpha-1})
			&=\frac{2\pi}{\Vol_{\omega }(X)}\sum_{\alpha =1}^l\mu_{\omega}(\mathcal{E}_{\alpha }/\mathcal{E}_{\alpha -1})(\pi_{\alpha }^{K}-\pi_{\alpha -1}^{K})\\
			&=\frac{2\pi}{\Vol_{\omega }(X)}\Phi_{\omega }^{HN}(E,\theta , K).
		\end{align*}
	By (\ref{lqcvg}), for every $0<q<\frac{2n}{n-1}$,
	\begin{equation}\label{370}
		\lim_{i\rightarrow \infty} \left\Vert\sqrt{-1}\Lambda_{\omega}F_{H_{\epsilon_i}, \theta }-\frac{2\pi}{\Vol_{\omega }(X)}\Phi^{HN}_\omega(E, \theta , K)\right\Vert_{L^q(K)}=0.
	\end{equation}
	The difference in \eqref{370} is uniformly bounded in $L^\infty$. For any $p>q$, interpolation gives
	$\|A_i\|_{L^p}\leq\|A_i\|_{L^\infty}^{1-q/p}\|A_i\|_{L^q}^{q/p}\to0$; the case $0<p\leq q$ follows from H\"older's inequality. This proves \eqref{eqthm1} for every finite $p>0$.
\end{proof}

\section{Reflexive Higgs sheaves over normal complex spaces}\label{higgsheaf-regular}
\begin{definition}
	A Higgs sheaf $(\mE_X,\theta_{X})$ consists of a reflexive coherent sheaf $\mE_X$ on a normal complex space $X$ and an $\mO_{X}$-linear morphism
\[
 \theta_{X_\reg}:\mE_{X_\reg}\longrightarrow
 \mE_{X_\reg}\otimes\Omega^1_{X_\reg}
\]
on $\mE_{X_\reg}:=\mE_X|_{X_\reg}$ with $\theta_{X_\reg}\wedge\theta_{X_\reg}=0.$ A coherent subsheaf of $\mE_X$ is called Higgs invariant if its restriction to $X_\reg$ is $\theta_{X_\reg}$-invariant. The terms torsion-free, reflexive and locally free refer to the underlying sheaf $\mE_X$ on $X$.
\end{definition}

In Section \ref{subsec:Higgs-extension}, we prove an extension theorem for Higgs fields on suitable orbifold modifications. Section \ref{subsec:stability-HN} recalls the notions of semistability and Harder--Narasimhan (HN) filtrations and establishes the finiteness and existence results needed later.

\subsection{Extension of Higgs fields over orbifold modifications}\label{subsec:Higgs-extension}

\begin{proposition}\label{lem-reflexive-pullback}\label{prop-higgs-orbifold-model}
	Let $X$ be a compact normal complex space of dimension $n\geq2$ with quotient singularities in codimension $2$, and let $\mE_X$ be a reflexive coherent sheaf equipped with a Higgs field $\theta_{X_\reg}$ on $\mE_{X_\reg}$. Then there exist a closed analytic subset $Z\subset X$ of codimension at least $3$, a compact complex orbifold $W_\orb$ with coarse space $W$, a projective bimeromorphic morphism $h:W\to X$, and a holomorphic orbi-bundle $E_\orb$ on $W_\orb$ such that, with $U=X\setminus Z$, the map $h$ is an isomorphism over $U$, the orbifold structure on $h^{-1}(U)$ is standard, and
	$$E_\orb|_{h^{-1}(U)}\cong(h_\orb^{[*]}\mE_X)|_{h^{-1}(U)}.$$
	Under this identification, the Higgs field induced by $\theta_{X_\reg}$ on $E_\orb|_{h^{-1}(U)}$ extends holomorphically to a Higgs field $\theta_\orb$ on $E_\orb$ over $W_\orb$.
\end{proposition}

We first prove an extension lemma for finite surjective morphisms.

\begin{lemma}\label{lem-spectral-form-extension}
	Let $X$ be a normal complex space with quotient singularities in codimension $2$, and let $\gamma:S\to X$ be finite and surjective, with $S$ normal and every component dominating $X$. For any resolution of singularities $\pi:\widetilde S\to S$, the pull-back of meromorphic differentials induces
	$$(\gamma^*\Omega_X^{[1]})^{\vee\vee}\longrightarrow\pi_*\Omega^1_{\widetilde S}.$$
\end{lemma}
\begin{proof}
	Choose $Z_0\subset X$ of codimension at least $3$ such that $X\setminus Z_0$ has only quotient singularities. Over a sufficiently small neighbourhood $B\subset X\setminus Z_0$, choose a finite quasi-\'etale chart $\kappa:V\to B$ with $V$ smooth. Put $S_B=\gamma^{-1}(B)$, and let $S'$ be the normalization of the dominating components of $S_B\times_BV$, with projections $p:S'\to S_B$ and $q_0:S'\to V$. Let $T$ be a resolution of the normalization of the dominating components of $S'\times_{S_B}\widetilde S_B$, where $\widetilde S_B=\pi^{-1}(S_B)$. We obtain the commutative diagram
	$$
	\begin{tikzcd}[row sep=large]
	T\arrow[r,"\rho"]\arrow[d,"\nu"']&\widetilde S_B\arrow[d,"\pi"]\\
	S'\arrow[r,"p"]\arrow[d,"q_0"']&S_B\arrow[d,"\gamma"]\\
	V\arrow[r,"\kappa"']&B.
	\end{tikzcd}
	$$
	Here $\nu$ is a resolution, and $\rho$ is proper, surjective and
	generically finite. A local section $\sigma$ of
	$(\gamma^*\Omega_X^{[1]})^{\vee\vee}$ pulls back, outside a subset of
	codimension at least $2$, to a section of the locally free sheaf
	$q_0^*\Omega_V^1$. Since $S'$ is normal, this section extends uniquely
	to a section
	$$\delta\in H^0(S',q_0^*\Omega_V^1).$$
	Set $\eta:=\nu^*\delta\in H^0(T,\Omega_T^1)$.

	Let $\tau$ be the meromorphic pull-back of $\sigma$ to
	$\widetilde S_B$. On a common dense open subset of $T$, we have
	$$\rho^*\tau=\eta.$$
	Suppose that $\tau$ has a pole along a prime divisor
	$D\subset\widetilde S_B$. Since $\rho$ is proper, surjective and
	generically finite, there is a prime divisor $D'\subset T$ dominating
	$D$. After an unramified localization, we may choose local coordinates such
	that $D=\{x=0\}$, $D'=\{t=0\}$, and
	$$x=t^e,$$
	with the tangential coordinates $y_l$ unchanged. For $k\geq1$,
	$$x^{-k}dx=e\,t^{-e(k-1)-1}dt,\qquad
	x^{-k}dy_l=t^{-ek}dy_l.$$
	Thus every polar term of $\tau$ gives a polar term of
	$\rho^*\tau$, contradicting the holomorphicity of $\eta$. Hence $\tau$
	has no divisorial poles and is holomorphic on $\widetilde S_B$. The
	local constructions agree on overlaps,
	so they induce a morphism
	$$(\gamma^*\Omega_X^{[1]})^{\vee\vee}\longrightarrow
	\pi_*\Omega_{\widetilde S}^1$$
	over $S\setminus\gamma^{-1}(Z_0)$.
	Since $\gamma$ is finite, $\gamma^{-1}(Z_0)$ has codimension at least
	$3$. Applying \cite[Lemma 24]{Nun24} with $p=1$ extends this morphism
	over $S$.
\end{proof}

We next construct a spectral cover over an arbitrary normal complex space to reduce the extension problem for Higgs fields to that for holomorphic $1$-forms. Spectral covers go back to Hitchin \cite{Hitchin87}; see also \cite{Heuer25} for smooth proper rigid analytic spaces over $p$-adic fields.

\begin{lemma}\label{lem-higgs-spectral-cover}
	Let $X$ be a normal complex space, and let $\mE_X$ be a reflexive coherent sheaf of positive rank $r$, equipped with a Higgs field $\theta_{X_\reg}$ on $\mE_{X_\reg}$. Contraction with $\theta_{X_\reg}$ defines a finite coherent commutative $\mO_X$-algebra
	$$\mathcal A=\image\bigl(\operatorname{Sym}_{\mO_X}\mT_X\longrightarrow\End(\mE_X)\bigr).$$
	Let $Y:=\operatorname{Spec}^{\mathrm{an}}_X(\mathcal A)$, and let
$S$ be the normalization of the union of the irreducible components
of $Y_{\mathrm{red}}$ that dominate $X$. Denote by
$\gamma:S\to X$ the induced morphism. Then $\gamma$ is finite and surjective, and there exists a canonical reflexive
one-form
$$
\sigma\in H^0\bigl(S,(\gamma^*\Omega_X^{[1]})^{\vee\vee}\bigr)
$$
with the following property. Over any sufficiently small connected open
subset $U$ of a suitable dense open subset of $X_{\reg}$ where $\mE_X$
is locally free and $\gamma$ is \'etale, the restrictions of $\sigma$
to the $d$ sheets of $\gamma$ correspond to holomorphic one-forms
$\lambda_1,\ldots,\lambda_d$ on $U$, where $d=\deg\gamma$.
There exist positive integers $r_1,\ldots,r_d$ with
$\sum_{a=1}^d r_a=r$ such that, for every local vector field $v$ on $U$,
$$
\det\bigl(T\Id-\theta_{X_{\reg}}(v)\bigr)
=
\prod_{a=1}^d\bigl(T-\lambda_a(v)\bigr)^{r_a}.
$$
Thus the characteristic coefficients of the Higgs field are symmetric
expressions in the $\lambda_a$, with each $\lambda_a$ counted with
multiplicity $r_a$.
\end{lemma}
\begin{proof}
	Contraction with $\theta_{X_\reg}$ extends by reflexivity to a morphism
$$
\varphi:\mT_X\longrightarrow\End(\mE_X),
\qquad
[\varphi(v),\varphi(w)]=0,
$$
where the commutativity follows from integrability.
Locally on $X$, choose a sufficiently small open neighbourhood $V$
and local generators
$$
\mO_V^{\oplus a}\twoheadrightarrow\mT_X|_V,
\qquad
(f_1,\ldots,f_a)\longmapsto\sum_{i=1}^a f_i v_i,
$$
and put $A_i:=\varphi(v_i)\in\End(\mT_X)$.
Let $V^\circ\subset V\cap X_\reg$ be the locus where $\mE_X$ is
locally free. Then $\codim_V(V\setminus V^\circ)\geq2$ and
thus the characteristic polynomials of $A_i$ on $V^\circ$ extend to
$$
P_i(T)=T^r+\sum_{k=0}^{r-1}c_{ik}T^k,
\qquad
c_{ik}\in H^0(V,\mO_V),
\qquad
P_i(T)|_{V^\circ}
=\det\bigl(T\Id-A_i|_{V^\circ}\bigr).
$$
By the Cayley--Hamilton theorem and torsion-freeness of $\End(\mE_X)$,
$$
P_i(A_i)|_{V^\circ}=0
\quad\Longrightarrow\quad
P_i(A_i)=0
\quad\text{on }V.
$$
Since the $A_i$ commute, it follows that
$$
\mathcal A|_V
=
\sum_{0\leq j_1,\ldots,j_a<r}
\mO_V\,A_1^{j_1}\cdots A_a^{j_a}
\subset\End(\mE_X)|_V
$$
is a finite coherent $\mO_X$-algebra. By \cite[Theorem 45 B.1]{Kaup83}, there is a unique complex space $(Y,f)$, up to isomorphism, finite over $X$ and equipped with an $\mO_X$-isomorphism $\mu:f_*\mO_Y\rightarrow \mathcal A$. This is the analytic spectrum of $\mathcal A$.

Since $r>0$, scalar endomorphisms give an inclusion
$$
\mO_X\hookrightarrow\mathcal A,
\qquad
f\longmapsto f\Id.
$$
Consequently, every irreducible component of $X$ is dominated by
an irreducible component of $Y_{\mathrm{red}}$. Since $Y\to X$ is
finite, the union of these dominating components maps surjectively
onto $X$. Normalization is finite, so $\gamma:S\to X$ is finite
and surjective.

The natural morphism $S\to Y$ induces an $\mO_S$-algebra morphism
$$
\chi:\gamma^*\mathcal A\longrightarrow\mO_S.
$$
Let $X^\circ\subset X_\reg$ be the locus where $\mE_X$ is locally
free. The composition of $\gamma^*\varphi$ with $\chi$ defines
$
\sigma^\circ\in
H^0\bigl(\gamma^{-1}(X^\circ),\gamma^*\Omega_{X^\circ}^1\bigr)$
by
$$
\langle\sigma^\circ,\gamma^*v\rangle
=\chi\bigl(\gamma^*\varphi(v)\bigr),\quad\forall v\in \mT_{X^\circ}.
$$
Since $\gamma$
is finite and every irreducible component of $S$ dominates $X$,
$$
\codim_S\bigl(S\setminus\gamma^{-1}(X^\circ)\bigr)\geq2.
$$
By reflexivity, there is a unique, canonical extension
$$
\sigma\in H^0\bigl(S,(\gamma^*\Omega_X^{[1]})^{\vee\vee}\bigr).
$$

Remove from $X^\circ$ the images of the nondominating components
of $Y_{\mathrm{red}}$ and of the locus where normalization is not
an isomorphism, together with the branch locus of $\gamma$.
On the resulting dense open subset,
$$
\mathcal A_{\mathrm{red}}\simeq\gamma_*\mO_S.
$$
For a sufficiently small connected open set $U$ in this locus,
$$
\gamma^{-1}(U)=\coprod_{a=1}^d U_a,
\qquad
\gamma|_{U_a}:U_a\xrightarrow{\sim}U,
\qquad
\mathcal A_{\mathrm{red}}|_U\simeq\mO_U^{\oplus d}.
$$
Let $\chi_a:\mathcal A|_U\to\mO_U$ be the character corresponding
to $U_a$, and let $\lambda_a$ be the one-form corresponding to
$\sigma|_{U_a}$. Then
$$
\lambda_a(v)=\chi_a\bigl(\varphi(v)\bigr).
$$
The standard idempotents of $\mO_U^{\oplus d}$ lift uniquely through
the locally nilpotent nilradical $\mathcal N$ of $\mathcal A|_U$
to idempotents $\varepsilon_1,\ldots,\varepsilon_d$ satisfying
$$
\varepsilon_a\varepsilon_b=\delta_{ab}\varepsilon_a,
\qquad
\sum_{a=1}^d\varepsilon_a=1.
$$
Thus
$$
\mE_X|_U=\bigoplus_{a=1}^d\mE_a,
\qquad
\mE_a:=\varepsilon_a(\mE_X|_U),
\qquad
r_a:=\rank\mE_a>0,
\qquad
\sum_{a=1}^d r_a=r.
$$
Each $\mE_a$ is locally free, and its rank is constant because $U$
is connected. The ranks are positive since
$\mathcal A|_U\hookrightarrow\End(\mE_X|_U)$ is faithful and
$\varepsilon_a\ne0$. For every local vector field $v$ on $U$,
$$
\varepsilon_a\bigl(\varphi(v)-\lambda_a(v)\Id\bigr)\in\mathcal N.
$$
Since $U$ is reduced, fiberwise nilpotence gives
$$
\varphi(v)|_{\mE_a}=\lambda_a(v)\Id+N_a(v),
\qquad
N_a(v)^{r_a}=0.
$$
Taking determinants in the direct sum decomposition gives
$$
\begin{aligned}
\det\bigl(T\Id-\theta_{X_\reg}(v)\bigr)
&=\prod_{a=1}^d
\det\bigl((T-\lambda_a(v))\Id-N_a(v)\bigr)\\
&=\prod_{a=1}^d\bigl(T-\lambda_a(v)\bigr)^{r_a}.
\end{aligned}
$$
The description of the characteristic coefficients follows.
\end{proof}

\begin{proof}[Proof of Proposition \ref{prop-higgs-orbifold-model}]
If $\mE_X=0$, apply Lemma~\ref{lemma-existence} and take the zero
orbi-bundle. Otherwise, let $\gamma:S\to X$ and $\sigma$ be the
spectral cover and spectral one-form constructed in
Lemma~\ref{lem-higgs-spectral-cover}. Choose a resolution
$\pi:\widetilde S\to S$. Since $X$ has quotient singularities in
codimension $2$, Lemma~\ref{lem-spectral-form-extension} induces
$$
H^0\bigl(S,(\gamma^*\Omega_X^{[1]})^{\vee\vee}\bigr)
\longrightarrow H^0(\widetilde S,\Omega_{\widetilde S}^1),
\qquad
\sigma\longmapsto\widetilde\sigma.
$$
Choose a partial orbifold resolution of $X$ as in
Lemma~\ref{lemma-existence}. Enlarge its exceptional image to a
closed analytic subset $Z\subset X$ such that
$$
\codim_X Z\geq3,
\qquad
U:=X\setminus Z,
$$
and the reflexive orbifold lift of $\mE_X$ is locally free over $U$.
On the standard orbifold charts over $U$, the induced Higgs field
extends across the inverse image of $X_{\rm sing}$ by Hartogs' theorem.

Principalize the ideal of $Z$ equivariantly so that its inverse
image has simple normal crossings. If a chart group permutes
intersecting branches, blow up the closed strata in decreasing
order of codimension. The resulting boundary components are
indexed by the original strata, and intersecting components have
distinct codimension labels. Thus distinct branches in the same
group orbit are disjoint. Denote the resulting morphism by
$$
f:Y_\orb\longrightarrow X.
$$
Its boundary
$$
D=\sum_j D_j
$$
has simple normal crossings, and each $D_j$ is a globally defined
smooth orbifold Cartier divisor. All these modifications are
supported over $Z$.

Set $d:=\deg\gamma$ and $m:=d!$. Taking the $m$-th root of each pair
$(\mO_{Y_\orb}(D_j),s_{D_j})$ gives
$$
q:Y'_\orb\longrightarrow Y_\orb.
$$
On a chart where $D=\{x_1\cdots x_b=0\}$, the corresponding root
chart is
$$
(s_1,\ldots,s_n)\longmapsto
(s_1^m,\ldots,s_b^m,s_{b+1},\ldots,s_n).
$$
The transition maps of the pairs glue these charts to a smooth
effective orbifold with the same coarse space. Locally, the lifted
chart groups are extensions of the original groups by $\mu_m^b$.

We next show that the characteristic coefficients of the
pulled-back Higgs field are holomorphic. Fix a general point of a
boundary divisor, and choose a chart $V$ of $Y_\orb$ with coordinates
$(x,y)$ such that the boundary is $\{x=0\}$. Let $S_V$ be the
normalization of the dominating components of $S\times_X V$.
Resolving the graph of the induced meromorphic map
$S_V\dashrightarrow\widetilde S$ and pulling back
$\widetilde\sigma$ shows that the induced spectral form on $S_V$
is holomorphic at every divisorial generic point.

After an unramified localization and shrinking $V$, a local branch
$B_e$ of $S_V$ admits coordinates $(t,y)$ in which
$$
B_e\longrightarrow V,
\qquad
(t,y)\longmapsto(t^e,y),
$$
where $e$ is the ramification index. Since $e\leq d$, we have
$e\mid m$. The root chart over $V$ is
$$
V'\longrightarrow V,
\qquad
(s,y)\longmapsto(s^m,y).
$$
Consequently,
$$
B_e\times_V V'=\{t^e=s^m\},
\qquad
t^e-s^m=\prod_{\zeta^e=1}(t-\zeta s^{m/e}).
$$
Its normalization is therefore the disjoint union of $e$
components, each mapped isomorphically onto $V'$.
Let $B'_e$ be the component defined by $t=s^{m/e}$.
The projections give the commutative diagram
$$
\begin{tikzcd}[column sep=large]
B'_e\arrow[r,"t=s^{m/e}"]\arrow[d,"\simeq"']
    &B_e\arrow[d,"x=t^e"]\\
V'\arrow[r,"x=s^m"']&V.
\end{tikzcd}
$$

Write the spectral form on $B_e$ as
$$
\sigma_e=a(t,y)\,dt+\sum_l b_l(t,y)\,dy_l.
$$
For each $\zeta^e=1$, the projection from the corresponding
normalized component to $B_e$ is
$$
\rho_\zeta(s,y)=(\zeta s^{m/e},y),
$$
and
$$
\rho_\zeta^*\sigma_e
=
\zeta\frac{m}{e}
a(\zeta s^{m/e},y)s^{m/e-1}\,ds
+\sum_l b_l(\zeta s^{m/e},y)\,dy_l.
$$
This is holomorphic since $m/e$ is a positive integer.
By Lemma~\ref{lem-higgs-spectral-cover}, the characteristic
coefficients are symmetric expressions in these spectral forms,
with their algebraic multiplicities. Hence they are holomorphic
at every divisorial generic point of the boundary. Away from the
boundary, they are holomorphic because they come from the existing
Higgs field. As meromorphic sections of locally free symmetric
differential sheaves, they extend holomorphically everywhere.

Let $\mE_{\orb,0}$ be the reflexive orbifold pull-back of $\mE_X$
to $Y'_\orb$, and put $r:=\rank\mE_X$. In a local cotangent frame
$\eta_1,\ldots,\eta_n$, write the meromorphic pulled-back Higgs field as
$$
\theta'_\orb=\sum_{l=1}^n A_l\eta_l,
\qquad
[A_i,A_j]=0.
$$
The characteristic coefficients of each $A_l$ are holomorphic.
By Cayley--Hamilton,
$$
\mathcal B:=\mO[A_1,\ldots,A_n]
=
\sum_{0\leq j_1,\ldots,j_n<r}
\mO\,A_1^{j_1}\cdots A_n^{j_n}
$$
is a finite coherent algebra. Set
$$
\mE'_\orb:=(\mathcal B\mE_{\orb,0})^{\vee\vee}.
$$
Since
$$
A_l(\mathcal B\mE_{\orb,0})
\subseteq\mathcal B\mE_{\orb,0},
$$
the Higgs field preserves $\mathcal B\mE_{\orb,0}$. Taking the
double dual yields a holomorphic morphism
$$
\theta'_\orb:\mE'_\orb
\longrightarrow\mE'_\orb\otimes\Omega^1_{Y'_\orb},
$$
because $\Omega^1_{Y'_\orb}$ is locally free.

Under a change of cotangent frame, the coefficients satisfy
$$
\widehat A_i=\sum_{j=1}^n c_{ij}A_j,
\qquad
(c_{ij})\in\operatorname{GL}_n(\mO),
$$
and hence
$$
\mO[\widehat A_1,\ldots,\widehat A_n]
=\mO[A_1,\ldots,A_n].
$$
Thus the construction is compatible with changes of charts and
chart-group actions, and glues equivariantly. Over $U$, the original Higgs field preserves
$\mE_{\orb,0}$, and therefore $\mE'_\orb=\mE_{\orb,0}$ there.

Applying Lemma~\ref{lem-resoorbi} to $\mE'_\orb$, we obtain
$$
g:W_\orb\longrightarrow Y'_\orb,
\qquad
E_\orb:=g^T\mE'_\orb
=g^*\mE'_\orb/\torsion,
$$
with $E_\orb$ locally free and all centers supported over $Z$.
Let
$$
p:g^*\mE'_\orb\twoheadrightarrow E_\orb
$$
be the quotient map. Since $E_\orb\otimes\Omega^1_{W_\orb}$ is
locally free, the composite
$(p\otimes dg)\circ g^*\theta'_\orb$ annihilates torsion and
induces $\theta_\orb$ in the commutative diagram
$$
\begin{tikzcd}[column sep=large]
g^*\mE'_\orb
  \arrow[r,"g^*\theta'_\orb"]
  \arrow[d,two heads,"p"']
&
g^*\mE'_\orb\otimes g^*\Omega^1_{Y'_\orb}
  \arrow[d,"p\otimes dg"]\\
E_\orb
  \arrow[r,"\theta_\orb"']
&
E_\orb\otimes\Omega^1_{W_\orb}.
\end{tikzcd}
$$

Let $h:W\to X$ be the morphism on coarse spaces induced by
$f\circ q\circ g$. It is projective and bimeromorphic, and all
modifications are supported over $Z$. The identity
$\theta_\orb\wedge\theta_\orb=0$ holds over the dense open subset
$h^{-1}(U\cap X_\reg)$ and hence everywhere, since its target is
locally free. Moreover,
$$
E_\orb|_{h^{-1}(U)}
\cong
(h_\orb^{[*]}\mE_X)|_{h^{-1}(U)},
$$
and the Higgs field agrees there with the one induced by
$\theta_{X_\reg}$. The orbifold structure over $h^{-1}(U)$ is
standard. This proves the proposition.
\end{proof}

\subsection{Semistability and HN filtrations}\label{subsec:stability-HN}

\begin{definition}[Slope]
	Let $\mF$ be a coherent sheaf of positive rank on a compact normal space $X$ of dimension $n$, and let $\alpha_0,\cdots,\alpha_{n-2}$ be nef classes on $X$. The {\em slope} of $\mF$ with respect to the polarization $(\alpha_0,\cdots,\alpha_{n-2})$ is defined by 
	\begin{equation}\label{equa-slope}
		\mu_{(\alpha_0,\cdots,\alpha_{n-2})}(\mF):=\frac{c_1(\mF)\cdot\alpha_0\cdots\alpha_{n-2}}{\rank \mF},
	\end{equation}
\end{definition}

\begin{definition}[Semistability]\label{defn-stable}
	Let $\mE_X$ be a torsion-free coherent sheaf on a compact normal space $X$ of dimension $n$, equipped with a Higgs field $\theta_{X_\reg}$ on $\mE_{X_\reg}$, and let $\alpha_0,\cdots,\alpha_{n-2}$ be nef classes. We say that $(\mE_X,\theta_{X_\reg})$ is {\em semistable with respect to the polarization $(\alpha_0,\cdots,\alpha_{n-2})$} if, for every coherent subsheaf $\mF_X\subset\mE_X$ with $0<\rank\mF_X<\rank\mE_X$ whose restriction to $X_\reg$ is $\theta_{X_\reg}$-invariant, we have
	\begin{equation}\label{equa-defn-stability}
		\mu_{(\alpha_0,\cdots,\alpha_{n-2})}(\mF_{X})\leq \mu_{(\alpha_0,\cdots,\alpha_{n-2})}(\mE_{X}).
	\end{equation}
\end{definition}

It suffices to test semistability on saturated coherent subsheaves on $X$.

\begin{lemma}\label{lem-compatable-existingnotion}
	A torsion-free Higgs sheaf $(\mE_X,\theta_{X_\reg})$ is $(\alpha_0,\cdots,\alpha_{n-2})$-semistable if and only if \eqref{equa-defn-stability} holds for every saturated coherent subsheaf $\mF_X\subset\mE_X$ with $0<\rank\mF_X<\rank\mE_X$ whose restriction to $X_\reg$ is $\theta_{X_\reg}$-invariant.
\end{lemma}

\begin{proof}
	The forward implication is immediate. Conversely, let $\mF_X\subset\mE_X$ be as in Definition~\ref{defn-stable} and take its saturation in $\mE_X$. Its restriction to $X_\reg$ is Higgs invariant by Lemma~\ref{lem-saturation-invariant}. The saturation has the same rank as $\mF_X$, and its quotient by $\mF_X$ is torsion, so its slope does not decrease by Lemma~\ref{lem-orbi-torsionsheaf}. The inequality \eqref{equa-defn-stability} for the saturation gives the result.
\end{proof}

We use the following finiteness result to construct the HN filtration.
\begin{proposition}\label{prop-slope-finiteness}
	Let $X$ be a compact normal K\"ahler space of dimension $n\geq2$ and $\mE$ a torsion-free coherent sheaf on $X$. Let $\alpha_0,\cdots,\alpha_{n-2}$ be nef and big classes on $X$ and set $\Omega=\alpha_0\cdots\alpha_{n-2}$. Then for any fixed constant $C\in\R$, the set
	$$\mathcal{A}_{C}:=\{\mu_{\Omega}(\mF):0\neq\mF\subset\mE\text{ is saturated},\ \mu_\Omega(\mF)\geq C\}$$
	is finite.
\end{proposition}

This follows from Toma's boundedness criterion \cite[Theorem 5.5 and Definition 2.8(3)]{Toma21}. We give an elementary proof.

\begin{proof}
	Take a projective resolution $\pi:Y\rightarrow X$ such that $E:=\pi^*\mE/\torsion$ is locally free, by resolving $X$ and then applying Lemma~\ref{lem-resoorbi}. For a saturated subsheaf $\mF\subset\mE$, let $\mS$ be the saturation of the image of $\pi^*\mF\rightarrow E$. Lemma~\ref{lemma-compatible-pullback} gives
	$$\mu_{\pi^*\Omega}(\mS)=\mu_\Omega(\mF).$$
	Since $Y$ is compact K\"ahler and each $\pi^*\alpha_i$ is nef and big, we may assume that $X$ is smooth and $\mE=E$ is locally free.

	Fix a K\"ahler form $\omega$ on $X$ and a smooth Hermitian metric $H$ on $E$. Choose $A>0$ such that the Griffiths curvature of $(\Lambda^qE,\Lambda^qH)$ is bounded above by $A\omega$ for all $1\leq q\leq\rank E$, with the normalization $\frac{\im}{2\pi}F_{\Lambda^qH}$. For a saturated subsheaf $\mF\subset E$ of rank $q$, the determinant map $\det\mF\rightarrow\Lambda^qE$ has no divisorial zeros. Principalizing its zero ideal gives a modification $f:\widehat X\rightarrow X$ and a line subbundle
	$$L\subset f^*\Lambda^qE,\qquad L=f^*\det\mF\otimes\mO_{\widehat X}(D),$$
	where $D$ is effective and $f$-exceptional. Indeed, saturation makes the determinant map nonvanishing at the generic point of every divisor, and $D$ is the common zero divisor of its pull-back. For the induced metric $h_L$, the subbundle curvature formula gives the smooth semipositive form
	$$\eta:=Af^*\omega-c_1(L,h_L)\geq0.$$
	Writing $\gamma:=A[\omega]-c_1(\mF)$, we obtain
	$$f^*\gamma=[\eta]+[D],\qquad\gamma=f_*[\eta].$$
	In particular, $\gamma$ has a closed positive current representative. For every pseudo-effective class $\tau$ and nef classes $\xi_1,\ldots,\xi_{n-2}$, the projection formula gives
	$$\gamma\cdot\tau\cdot\xi_1\cdots\xi_{n-2}
	=[\eta]\cdot f^*\tau\cdot f^*\xi_1\cdots f^*\xi_{n-2}\geq0.$$
	The exceptional term vanishes because all the other factors are pulled back from $X$. The inequality follows by approximating the nef factors by K\"ahler classes and pairing with a positive current in $f^*\tau$.

	Choose $a_i>0$ such that $\alpha_i-a_i[\omega]$ is pseudo-effective. Applying the preceding inequality to replace the factors of $\Omega$ one at a time, we obtain
	$$\gamma\cdot\Omega\geq a\,\gamma\cdot[\omega]^{n-1},\qquad a:=\prod_{i=0}^{n-2}a_i>0.$$
	If $\mu_\Omega(\mF)\geq C$, then $c_1(\mF)\cdot\Omega\geq b$, where $b:=\min\{C,(\rank E)C\}$. Hence
	$$0\leq\gamma\cdot[\omega]^{n-1}\leq\frac{A[\omega]\cdot\Omega-b}{a}.$$
	Thus the positive currents $f_*\eta$ have uniformly bounded mass, so their cohomology classes $\gamma$ are bounded in $H^2(X,\R)$. Indeed, pairing with smooth representatives of a basis of $H^{2n-2}(X,\R)$ is bounded by a constant times the mass. The classes $c_1(\mF)$ are therefore bounded and lie in the discrete image of $H^2(X,\mathbb Z)$ in $H^2(X,\R)$. Only finitely many such classes occur. Since the ranks also take finitely many values, $\mathcal A_C$ is finite. On the original space, Lemma~\ref{lemma-compatible-pullback} gives $c_1(\mF)=\pi_*c_1(\mS)$. Thus the proof also shows that the homology classes $c_1(\mF)$ of these saturated subsheaves take only finitely many values.
\end{proof}

We can now construct the HN filtration for mixed nef and big polarisations on compact normal K\"ahler spaces.
\begin{proposition}[Saturated subsheaf with the maximal slope]\label{lem-maximal-unique}
	Let $\mE_X$ be a torsion-free coherent sheaf on a compact normal K\"ahler space $X$, equipped with a Higgs field $\theta_{X_\reg}$ on $\mE_{X_\reg}$, and let $\alpha_0,\ldots,\alpha_{n-2}$ be nef and big classes. Then there exists a saturated subsheaf $\mE_{1}$ of $\mE_X$ such that $\mE_{X_\reg,1}=\mE_1|_{X_\reg}$ is $\theta_{X_\reg}$-invariant and, for every nonzero coherent subsheaf $\mF_X\subset\mE_X$ whose restriction to $X_\reg$ is $\theta_{X_\reg}$-invariant, we have
	\begin{itemize}
		\item $\mu_{(\alpha_0,\cdots,\alpha_{n-2})}({\mF_X})\leq\mu_{\alpha_0,\cdots,\alpha_{n-2}}(\mE_{1})$.
		\item $\rank(\mF_X)\leq \rank(\mE_1)$ if $\mu_{(\alpha_0,\cdots,\alpha_{n-2})}({\mF_X})=\mu_{\alpha_0,\cdots,\alpha_{n-2}}(\mE_{1})$.
	\end{itemize}
	In particular, $\mE_1$ with the induced Higgs field on $X_\reg$ is $(\alpha_0,\cdots,\alpha_{n-2})$-semistable.
\end{proposition}

\begin{proof}
	Put $\Omega=\alpha_0\cdots\alpha_{n-2}$. Saturation preserves Higgs invariance by Lemma~\ref{lem-saturation-invariant}, applied on $X_\reg$, and does not decrease the slope by Lemma~\ref{lem-orbi-torsionsheaf}. Since $\mE_X$ itself is an admissible subsheaf, Proposition~\ref{prop-slope-finiteness}, applied with the lower bound $\mu_\Omega(\mE_X)$, shows that the maximal slope $M$ is attained. Choose a saturated invariant subsheaf $\mE_1$ of slope $M$ and maximal rank.

	The uniqueness of $\mE_1$ follows by the same argument as in the proof of \cite[Lemma 5.7.17]{Koba}, restricted to Higgs-invariant subsheaves. Its semistability follows from the maximality of its slope.
\end{proof}

By applying Proposition~\ref{lem-maximal-unique} to $\mE_X/\mE_1$ with its induced Higgs field on $X_\reg$ and proceeding by induction, we obtain
\begin{proposition}[HN filtration]\label{defn-HN-existence}
	Let $\mE_X$ be a torsion-free coherent sheaf of rank $r$ on a compact normal K\"ahler space $X$, equipped with a Higgs field $\theta_{X_\reg}$ on $\mE_{X_\reg}$. For any $n-1$ nef and big classes $\alpha_0,\ldots,\alpha_{n-2}$, $\mE_X$ admits a unique filtration of saturated coherent subsheaves
	$$0=\mE_{0}\subsetneq \mE_{1}\subsetneq\cdots\subsetneq\mE_{l}=\mE_{X}$$
	such that $\mE_{X_\reg,k}:=\mE_k|_{X_\reg}$ is $\theta_{X_\reg}$-invariant, each quotient sheaf $\mQ_k=\mE_k/\mE_{k-1}$ is torsion-free and $\mu_{\alpha_0,\cdots,\alpha_{n-2}}(\mQ_k)>\mu_{\alpha_0,\cdots,\alpha_{n-2}}(\mQ_{k+1})$, and each $\mQ_k$ with its induced Higgs field on $X_\reg$ is $(\alpha_0,\cdots,\alpha_{n-2})$-semistable.
\end{proposition}

\begin{proof}
	Apply the induction in Lemma~\ref{lem-orbi-HNfiltration}, using Proposition~\ref{lem-maximal-unique} in place of Lemma~\ref{maximalslope}.
\end{proof}

\begin{definition}[HN type]\label{defn-HN-type}
	In the setting of Proposition~\ref{defn-HN-existence}, put $\Omega=\alpha_0\cdots\alpha_{n-2}$. The {\em HN type} is
	\[
	 \vec{\mu}_{\Omega}(\mE_X,\theta_{X_\reg})
	 =(\mu_{1,\Omega},\ldots,\mu_{r,\Omega}),
	\]
	where $\mu_{i,\Omega}:=\mu_{\Omega}(\mQ_k)$
	for $\rank(\mE_{k-1})+1\leq i\leq \rank(\mE_k)$.
\end{definition}

The same finiteness argument controls the HN filtration under a small perturbation.

\begin{proposition}[The stability of HN filtration]\label{prop-stable-HNfiltration}
	Let $\mE_X$ be a torsion-free coherent sheaf on a compact K\"ahler space $X$ of dimension $n$, equipped with a Higgs field $\theta_{X_\reg}$ on $\mE_{X_\reg}$. Let $\alpha_{0},\cdots,\alpha_{n-2}$ be nef and big classes and $\beta$ a nef class on $X$, and set $\Omega_\epsilon:=(\alpha_0+\epsilon\beta)\cdots(\alpha_{n-2}+\epsilon\beta)$ for $\epsilon>0$. The Harder-Narasimhan filtration of $(\mE_X,\theta_{X_\reg})$ with respect to $\Omega_\epsilon$ is independent of $\epsilon$ for $0<\epsilon\ll1$. Moreover,
	$$\vec\mu_{\Omega_\epsilon}(\mE_X,\theta_{X_\reg})
	\longrightarrow\vec\mu_{\Omega_0}(\mE_X,\theta_{X_\reg})\qquad(\epsilon\to0^+).$$
\end{proposition}

\begin{proof}
	Write $\Omega=\Omega_0$, and let $\mF_\epsilon$ be the first HN subsheaf for $\Omega_\epsilon$. Since
	$\mu_{\Omega_\epsilon}(\mF_\epsilon)\geq\mu_{\Omega_\epsilon}(\mE_X)$, the determinant-curvature estimate in the proof of Proposition~\ref{prop-slope-finiteness}, applied to the mixed products occurring in $\Omega_\epsilon-\Omega$, gives a constant $C>0$, independent of $\mF_X$ and $0<\epsilon\leq1$, such that
	\begin{equation}\label{equa-fixed-limiting}
	\mu_{\Omega_\epsilon}(\mF_X)-\mu_\Omega(\mF_X)\leq C\epsilon.
	\end{equation}
	Since $\mu_{\Omega_\epsilon}(\mE_X)$ is bounded below for $0<\epsilon\leq1$, this gives a fixed lower bound for $\mu_\Omega(\mF_\epsilon)$. The proof of Proposition~\ref{prop-slope-finiteness} shows that the saturated subsheaves above this lower bound have only finitely many ranks and first Chern classes. Hence their slope functions
	$$P_{\mF}(\epsilon)=\frac{c_1(\mF)\cdot\prod_{i=0}^{n-2}(\alpha_i+\epsilon\beta)}{\rank\mF}$$
	form a finite set of real polynomials.

	Choose $\epsilon_0>0$ so that every nonzero difference of these polynomials has constant sign on $0<\epsilon<\epsilon_0$. Their ordering is then fixed. Since every first HN subsheaf belongs to this family, the family of maximal-slope subsheaves of maximal rank is independent of $\epsilon$. Uniqueness in Proposition~\ref{lem-maximal-unique} identifies the common first HN subsheaf. Repeating the argument for the fixed quotient gives, by induction on the rank, a common filtration
	$$0=\mE_0\subsetneq\mE_1\subsetneq\cdots\subsetneq\mE_l=\mE_X$$
	for all sufficiently small positive $\epsilon$.

	Set $\mQ_i=\mE_i/\mE_{i-1}$. For each fixed invariant subsheaf of $\mQ_i$, pass to the limit in its $\Omega_\epsilon$-semistability inequality. Thus $\mQ_i$ is $\Omega$-semistable, and the limiting slopes of these quotients are weakly decreasing. Merge consecutive quotients of the same limiting slope. Each merged quotient is semistable: intersect any invariant subsheaf with its filtration, and add the degree inequalities for the resulting subquotients. The merged filtration is therefore the $\Omega$-HN filtration. Repeating each quotient slope according to its rank proves the asserted convergence.
\end{proof}

The corresponding statements hold for complex K\"ahler orbifolds.

\begin{proposition}\label{prop-orbifold-stability}
	Let $X_\orb$ be a complex K\"ahler orbifold and $X$ its quotient space. Suppose that $(\mE_\orb,\theta_\orb)$ is a torsion-free Higgs orbi-sheaf and $\alpha_0,\cdots,\alpha_{n-2},\beta$ are nef and big classes on $X_\orb$. Set $\Omega_\epsilon:=\{\alpha_0+\epsilon\beta,\cdots,\alpha_{n-2}+\epsilon\beta\}$. We have
	\begin{itemize}
		\item[(1)] If $(\mE_\orb,\theta_\orb)$ is $\Omega$-stable, then $(\mE_\orb,\theta_\orb)$ is $\Omega_\epsilon$-stable for $0<\epsilon\ll1$.
		\item[(2)] The Harder-Narasimhan filtration of $(\mE_\orb,\theta_\orb)$ with respect to $\Omega_\epsilon$ is independent of $\epsilon$ for $0<\epsilon\ll1$.
	\end{itemize}
\end{proposition}

Recall the following elementary statement.
\begin{lemma}[cf. {\cite[Lemma 3.5]{DO23}}]\label{wulemma-kahler}
	Let $X_\orb$ be a compact complex orbifold and $X$ its quotient space. A class $\alpha\in H^2(X,\R)$ contains a K\"ahler (resp. nef, big) form on $X$ if and only if it contains an orbifold K\"ahler (resp. nef, big) form on $X_\orb$.
\end{lemma}

\begin{proof}[Proof of Proposition \ref{prop-orbifold-stability}]
	Let $\mE_X=\{((\pi_i)_*\mE_i)^{G_i}\}$. A saturated invariant orbi-subsheaf $\mF_\orb\subset\mE_\orb$ descends to a saturated subsheaf $\mF_X\subset\mE_X$: invariant push-forward is exact and preserves torsion-freeness by Lemma~\ref{push-forward}. Fix a finite orbifold atlas of $X_\orb$, and let $q$ be the least common multiple of the orders of its local groups. Then Lemma~\ref{descent} gives
	$$c_1^\orb(\mF_\orb)=c_1(\mF_X)+\sum_{j=1}^s d_jc_1^\orb(E_{\orb,j}),
	\qquad 0\leq d_j\leq(\rank\mE_\orb)(q-1).$$
	There are only finitely many possible tuples $(\rank\mF_\orb,d_1,\ldots,d_s)$. Thus a lower bound for $\mu_\Omega(\mF_\orb)$ gives a uniform lower bound for $\mu_\Omega(\mF_X)$. By Proposition~\ref{prop-slope-finiteness} and its proof, only finitely many ranks and first Chern classes occur downstairs. The displayed formula gives the required finiteness of orbifold first Chern classes upstairs.

	The one-sided estimate \eqref{equa-fixed-limiting} also holds on the orbifold. Indeed, take a fixed locally free model by Lemma~\ref{lem-resoorbi}, use invariant Hermitian metrics, and apply the determinant curvature argument from Proposition~\ref{prop-slope-finiteness} on its charts. Principalization may be chosen functorially, so the construction is compatible with the finite groups and the chart embeddings. The resulting positive current pairs nonnegatively with each mixed nef coefficient. The proof of Proposition~\ref{prop-stable-HNfiltration} now gives (1). The maximal-slope construction and the finite-polynomial argument in Proposition~\ref{prop-stable-HNfiltration} apply to invariant orbi-subsheaves and their quotients, giving (2) and convergence of the HN types as $\epsilon\to0^+$.
\end{proof}

\subsection{Proof of Corollary \ref{semistable-tensorproduct}}\label{subsec:model-comparison}

We compare the Higgs orbi-bundle constructed in Proposition~\ref{prop-higgs-orbifold-model} with the original Higgs sheaf.

\begin{proposition}[Comparison of HN filtrations and orbifold Chern classes]\label{prop-higgs-model-comparison}
	In the setting of Proposition~\ref{prop-higgs-orbifold-model}, let $h:W\to X$, $W_\orb$, $(E_\orb,\theta_\orb)$ and $U=X\setminus Z$ be the data constructed there. Then the following statements hold.
	\begin{enumerate}[label=(\arabic*)]
		\item Let $\alpha_0,\ldots,\alpha_{n-2}$ be nef and big classes on $X$, and set $\Omega=\alpha_0\cdots\alpha_{n-2}$. The $h^*\Omega$-HN filtration of $(E_\orb,\theta_\orb)$ is the saturated transform of the $\Omega$-HN filtration of $(\mE_X,\theta_{X_\reg})$. Over $h^{-1}(U)$, its terms are the reflexive orbifold pull-backs of the corresponding terms downstairs. In particular,
		$$\vec\mu_{h^*\Omega}(E_\orb,\theta_\orb)
		=\vec\mu_\Omega(\mE_X,\theta_{X_\reg}).$$
		\item The orbifold Chern classes satisfy
		$$h_*\bigl(c_i^\orb(E_\orb)\cap[W]\bigr)
		=\widehat c_i(\mE_X)\quad\text{in }H_{2n-2i}(X,\R),\qquad i=1,2,$$
		and
		$$h_*\bigl(c_1^\orb(E_\orb)^2\cap[W]\bigr)
		=\widehat c_1(\mE_X)^2\quad\text{in }H_{2n-4}(X,\R).$$
		Consequently, for every $\Gamma\in H^{2n-4}(X,\R)$,
		$$\Delta^\orb(E_\orb)\cdot h^*\Gamma
		=\widehat\Delta(\mE_X)\cdot\Gamma.$$
	\end{enumerate}
\end{proposition}

\begin{proof}
	The construction in Proposition~\ref{prop-higgs-orbifold-model} gives a generically isomorphic morphism
	$$\psi:h_\orb^*\mE_X\longrightarrow E_\orb.$$
	It comes from the inclusion of the initial reflexive pull-back lattice in the Higgs-invariant lattice, followed by pull-back and passage to the torsion-free quotient. Over $h^{-1}(U)$, the double dual of $\psi$ gives the prescribed identification with the reflexive orbifold pull-back.

	The saturated transform of a coherent subsheaf of $\mE_X$ depends only on the induced meromorphic identification: clear denominators locally, take the image in $E_\orb$, and saturate. Uniqueness of saturation makes this independent of the denominators and ensures that the subsheaves glue equivariantly.

	For a saturated coherent subsheaf $\mF_X\subset\mE_X$, let $\mF_\orb\subset E_\orb$ be the saturation of the image of $h_\orb^*\mF_X$ under $\psi$. Conversely, for a saturated orbi-subsheaf $\mF_\orb\subset E_\orb$, take its invariant push-forward to $W$, followed by its coherent direct image under $h$. Its inclusion over $U$ into $\mE_X|_U$ extends across $Z$ by reflexivity of $\mE_X$. Saturating the image in $\mE_X$ gives $\mF_X$. The two constructions preserve ranks and are inverse: they agree over the smooth locus in $U$, and a saturated coherent subsheaf of a torsion-free sheaf is determined by its restriction to a dense open set. Over the standard orbifold on $U$, reflexivity identifies $\mF_\orb$ with the reflexive orbifold pull-back of $\mF_X$. Higgs invariance is preserved because the fields agree on this dense open set and the quotients by the saturated subsheaves are torsion-free.

	For corresponding subsheaves, Lemma~\ref{descent} identifies their first Chern classes on the standard orbifold over $U$. Since $\codim_XZ\geq3$, the homology restriction sequence gives
	$$h_*\bigl(c_1^\orb(\mF_\orb)\cap[W]\bigr)=c_1(\mF_X).$$
	The projection formula therefore yields
	$$\mu_{h^*\Omega}(\mF_\orb)=\mu_\Omega(\mF_X).$$
	The same holds for the ambient sheaves. Taking inverse images also gives the correspondence for invariant subsheaves of successive quotients. Hence the saturated transforms of the HN terms downstairs have semistable quotients with the same strictly decreasing slopes and form the HN filtration upstairs. This proves (1).

	The Chern class identities in (2) follow from the construction in Definition~\ref{defn-orbifold}, since $E_\orb$ agrees with the reflexive orbifold pull-back over $h^{-1}(U)$ and $\codim_XZ\geq3$. The discriminant identity follows by the projection formula.
\end{proof}

The $L^p$-approximate critical Hermitian metrics and Corollary \ref{coro-characterization} give the following slope estimates (see \cite[Section 4.1]{LZZ21}).
\begin{lemma}\label{orbifold-HNtype-Calculus}
	Let $(E,\theta_E)$ and $(F,\theta_F)$ be Higgs bundles on a compact Gauduchon manifold $(X,\omega)$. For $1\leq p\leq\rank E$,
	\[
	 \mu_{\min}(\Lambda^pE)\geq p\mu_{\min}(E),\qquad
	 \mu_{\max}(\Lambda^pE)\leq p\mu_{\max}(E).
	\]
	For $p\geq1$,
	\[
	 \mu_{\min}(S^pE)=p\mu_{\min}(E),\qquad
	 \mu_{\max}(S^pE)=p\mu_{\max}(E).
	\]
	Moreover,
	\[
	 \mu_{\min}(E\otimes F)=\mu_{\min}(E)+\mu_{\min}(F),\qquad
	 \mu_{\max}(E\otimes F)=\mu_{\max}(E)+\mu_{\max}(F),
	\]
	where the tensor-product Higgs field is
	$\theta_E\otimes\Id_F+\Id_E\otimes\theta_F$.
\end{lemma}

These estimates extend to singular spaces as follows.
\begin{proposition}
	Let $\mE_X$ and $\mF_X$ be torsion-free coherent sheaves on a compact normal K\"ahler space $X$ with quotient singularities in codimension $2$, equipped with Higgs fields on $X_\reg$. Let $\alpha_0,\cdots,\alpha_{n-2}$ be nef and big classes and set $\Omega=\alpha_0\cdots\alpha_{n-2}$. The conclusions of Lemma~\ref{orbifold-HNtype-Calculus} hold for these sheaves, with tensor, symmetric and exterior powers taken modulo torsion and followed by the double dual.
\end{proposition}
\begin{proof}
	Passing to the double dual preserves the relevant slopes. Apply the construction in Proposition~\ref{prop-higgs-orbifold-model} to $\mE_X^{\vee\vee}\oplus\mF_X^{\vee\vee}$. Its canonical projections commute with the Higgs action and preserve the initial lattice, hence every lattice in the construction. Resolving the direct sum gives Higgs orbi-bundles $(E_\orb,\theta_E)$ and $(F_\orb,\theta_F)$ on one compact K\"ahler orbifold $h:W_\orb\to X$. Let $U=X\setminus Z$ be the unchanged open set of this construction.

	Fix an orbifold K\"ahler form $\eta$ and put
	$$\Omega_\epsilon=\prod_{j=0}^{n-2}(h^*\alpha_j+\epsilon[\eta]).$$
	For $\epsilon>0$, apply Lemma~\ref{orbifold-HNtype-Calculus} with the Gauduchon metric determined by these mixed K\"ahler forms. Its proof uses the orbifold approximate critical metrics, so it also holds on $W_\orb$. By Proposition~\ref{prop-orbifold-stability}, all the HN types converge as $\epsilon\to0^+$. The comparison of saturated subsheaves and degrees in Proposition~\ref{prop-higgs-model-comparison} applies to these bundles and their tensor, symmetric and exterior powers. Indeed, the construction gives compatible equivariant meromorphic identifications with the corresponding sheaves downstairs, inducing the prescribed Higgs isomorphisms over $h^{-1}(U)$. The limiting slopes are therefore those on $X$, which proves the asserted identities and inequalities.
\end{proof}

Semistability is equivalent to equality of the maximal, average, and minimal slopes. The preceding proposition therefore proves Corollary~\ref{semistable-tensorproduct}.

\section{Orbifold Bogomolov-Gieseker and Miyaoka-Yau inequalities}\label{BMY}

\subsection{Proof of Theorem \ref{singulargeneralized}}
We first prove a quantitative mixed inequality for K\"ahler classes. Take the Higgs orbifold model
$$h:W_\orb\longrightarrow X,\qquad(E_\orb,\theta_\orb)$$
from Proposition~\ref{prop-higgs-orbifold-model}. Proposition~\ref{prop-higgs-model-comparison} gives, for
$\Omega=\alpha_0\cdots\alpha_{n-2}$ and $\Gamma=\alpha_1\cdots\alpha_{n-2}$
with each $\alpha_i$ nef and big (and $\Gamma=1$ when $n=2$),
\begin{equation}\label{singular-pullback-data}
 \widehat\Delta(\mE_X)\cdot\Gamma
 =\Delta^\orb(E_\orb)\cdot h^*\Gamma,
 \qquad
 \vec\mu_{h^*\Omega}(E_\orb,\theta_\orb)
 =\vec\mu_\Omega(\mE_X,\theta_{X_\reg}).
\end{equation}

Let $\omega_0,\ldots,\omega_{n-2}$ be K\"ahler forms on $X$ and let
$\eta$ be an orbifold K\"ahler form on $W_\orb$. Put
\[
 \omega_{k,\varepsilon}=h^*\omega_k+\varepsilon\eta,
 \qquad
 \Omega_\varepsilon=\prod_{k=0}^{n-2}[\omega_{k,\varepsilon}],
\]
and let $\omega_\varepsilon^{n-1}=\omega_{0,\varepsilon}\wedge\cdots
\wedge\omega_{n-2,\varepsilon}$. Corollary~\ref{coro-astheno} gives
\begin{align}
 \Delta^\orb(E_\orb)\cdot
 \prod_{k=1}^{n-2}[\omega_{k,\varepsilon}]
 \geq-\frac{n}{\left(\int_{W_\orb}\omega_\varepsilon^n\right)^2}
 \left(\int_{W_\orb}\frac{\omega_\varepsilon^n}
 {\tr_{\omega_\varepsilon}\omega_{0,\varepsilon}}\right)
 \sum_{a=1}^r
 \bigl(\mu_{\Omega_\varepsilon,a}-\mu_{\Omega_\varepsilon}(E_\orb)\bigr)^2,
 \label{singular-mixed-perturbed}
\end{align}
The coefficient of the sum is uniformly bounded and converges as
$\varepsilon\to0$. Indeed, the K\"ahler
forms $\omega_k$ are uniformly mutually comparable on $X$. Hence the forms
$\omega_{k,\varepsilon}$ are uniformly comparable to
$h^*\omega_0+\varepsilon\eta$. The cofactor-matrix description of the
Michelsohn root gives the same uniform comparison for $\omega_\varepsilon$.
It follows that $\tr_{\omega_\varepsilon}\omega_{0,\varepsilon}$ is bounded
above and below independently of $\varepsilon$. Away from the exceptional
locus, $\omega_\varepsilon$ converges smoothly to $h^*\omega$, where
$\omega^{n-1}=\omega_0\wedge\cdots\wedge\omega_{n-2}$. Moreover, every term
of $(h^*\omega_0+\varepsilon\eta)^n$ containing $\eta$ has total mass
$O(\varepsilon)$, while $(h^*\omega_0)^n$ gives no mass to the exceptional
locus. Dominated convergence therefore gives the limit
\[
 \frac{n}{\left(\int_{X_\reg}\omega^n\right)^2}
 \int_{X_\reg}\frac{\omega^n}{\tr_\omega\omega_0}.
\]
By Proposition~\ref{prop-orbifold-stability} and the HN comparison in
Proposition~\ref{prop-higgs-model-comparison}, the HN types converge to
that of $(\mE_X,\theta_{X_\reg})$. Letting $\varepsilon\to0$ in
\eqref{singular-mixed-perturbed} proves the quantitative mixed inequality
on $X$, and in particular \eqref{Bo3-mixed} when the sheaf is
$\Omega$-semistable.

It remains to prove \eqref{Bo3}. Choose a K\"ahler class $[\omega_X]$ and
set $\alpha_\varepsilon=\alpha+\varepsilon[\omega_X]$. Apply the preceding
inequality with every $\omega_k$ representing $\alpha_\varepsilon$. By
Remark~\ref{mixed-normalisation-check},
\[
 \widehat\Delta(\mE_X)\cdot\alpha_\varepsilon^{n-2}
 \geq-\frac1{\alpha_\varepsilon^n}
 \sum_{a=1}^r
 \bigl(\mu_{\alpha_\varepsilon,a}
       -\mu_{\alpha_\varepsilon}(\mE_X)\bigr)^2.
\]
The HN filtration is constant for sufficiently small $\varepsilon$ by
Proposition~\ref{prop-stable-HNfiltration}; hence its slopes converge to the
$\alpha^{n-1}$-slopes. Since $\alpha$ is big, $\alpha^n>0$, and taking the
limit proves \eqref{Bo3}.

\subsection{The Miyaoka-Yau inequality}\label{miyaoka-yau}

Throughout this subsection, let $X$ be a compact K\"ahler klt space of dimension $n$ with $K_X$ nef.

\subsubsection{The natural Higgs sheaf and a subsheaf on a resolution}

\begin{definition}[The natural Higgs sheaf]\label{defn-naturalsheaf-revised}
The natural Higgs sheaf has underlying reflexive coherent sheaf
$\mE_X:=\Omega_X^{[1]}\oplus\mO_X$ on $X$. Its Higgs field is defined on
$\mE_{X_\reg}=\Omega^1_{X_\reg}\oplus\mO_{X_\reg}$ by
\[
 \theta_{X_\reg}(a,b):=\left(0,\frac{a}{\sqrt{n+1}}\right).
\]
Here the right-hand side is viewed as a section of
$(\Omega^1_{X_\reg}\oplus\mO_{X_\reg})\otimes\Omega^1_{X_\reg}$.
Lemma~\ref{calculus} gives
\begin{equation}\label{natural-discriminant}
 \widehat\Delta(\mE_X)
 =2\widehat c_2(\mT_X)-\frac{n}{n+1}\widehat c_1(\mT_X)^2.
\end{equation}
\end{definition}

Recall the following elementary statement.
\begin{lemma}\label{lem-upper-low}
	In the setting of Definition \ref{defn-HN-type},
	$$-(r-1)(\mu_{1,\Omega}-\mu_\Omega(\mE_X))\leq\mu_{i,\Omega}-\mu_\Omega(\mE_X)\leq\mu_{1,\Omega}-\mu_\Omega(\mE_X).$$
\end{lemma}
\begin{proof}
	Repeat each quotient slope according to its rank. The assertion follows from
	$\sum_{i=1}^r(\mu_{i,\Omega}-\mu_\Omega(\mE_X))=0$ and the decreasing order of the HN slopes.
\end{proof}

Fix a K\"ahler form $\omega_X$ and put
$\alpha_\delta=K_X+\delta[\omega_X]$.  Combining \eqref{Bo3} with
Lemma~\ref{lem-upper-low}, we obtain
\begin{equation}\label{estimate}
 \widehat\Delta(\mE_X)\cdot\alpha_\delta^{n-2}
 \geq-\frac{C_n}{\alpha_\delta^n}
 \bigl(\mu_{1,\alpha_\delta}(\mE_X)
       -\mu_{\alpha_\delta}(\mE_X)\bigr)^2.
\end{equation}

Choose a log resolution $h:\widetilde X\to X$ which is an isomorphism
over $X_\reg$, and set $U=h^{-1}(X_\reg)$. On $\widetilde X$ consider
the natural Higgs bundle
$(\mE_{\widetilde X},\theta_{\widetilde X})
=(\Omega^1_{\widetilde X}\oplus\mO_{\widetilde X},
\theta_{\widetilde X})$.

For each $\delta>0$, let $\mF_{\delta,X}\subset\mE_X$ be the maximal
destabilising Higgs subsheaf for $\alpha_\delta^{n-1}$.
Lemma~\ref{lem-spectral-form-extension}, applied to the identity map of
$X$, and adjunction give a generically isomorphic morphism
$$h^*\mE_X\longrightarrow\mE_{\widetilde X}.$$
Over $U$, it is the canonical identification $dh\oplus\Id$ and
intertwines the Higgs fields. Let $\mF_\delta\subset\mE_{\widetilde X}$
be the saturation of the image of $h^*\mF_{\delta,X}$.
Higgs invariance holds over $U$ and extends everywhere because
$\mE_{\widetilde X}/\mF_\delta$ is torsion-free.
Lemma~\ref{lemma-compatible-pullback} and the projection formula give
$$\mu_{h^*\alpha_\delta}(\mF_\delta)
=\mu_{\alpha_\delta}(\mF_{\delta,X}),\qquad
\mu_{h^*\alpha_\delta}(\mE_{\widetilde X})
=\mu_{\alpha_\delta}(\mE_X).$$
Hence
\begin{equation}\label{natural-slope-gap}
 \mu_{h^*\alpha_\delta}(\mF_\delta)
 -\mu_{h^*\alpha_\delta}(\mE_{\widetilde X})
 =\mu_{1,\alpha_\delta}(\mE_X)-\mu_{\alpha_\delta}(\mE_X).
\end{equation}
Since $X$ is klt,
\begin{equation}\label{adjunction}
 K_{\widetilde X}=h^*K_X+D,\qquad
 D=\sum_{i=1}^N a_iD_i,\qquad a_i>-1,
\end{equation}
where the support of $D$ has simple normal crossings.

\begin{proposition}[Slope estimate]\label{upperslope}
There exists a constant $C>0$, depending on the fixed resolution, the
background metrics, the numbers $|a_i|$, and $n$, but not on $\delta$, such
that
\begin{equation}\label{slope-estimate-main}
 \mu_{h^*\alpha_\delta}(\mF_\delta)
 -\mu_{h^*\alpha_\delta}(\mE_{\widetilde X})
 \leq C\delta[\omega_X]\cdot\alpha_\delta^{n-1}.
\end{equation}
Consequently $(\mE_X,\theta_{X_\reg})$ is semistable with respect to
$K_X^{n-1}$.
\end{proposition}

The slope estimate in this proposition may be viewed as a quantitative
version of the semistability result of Guenancia--Taji \cite{GT22}.

\begin{proof}
Fix $\delta,t>0$ and set
\[
 \beta_{\delta,t}=2\pi\delta h^*\omega_X+t\omega_{\widetilde X},
 \qquad A_{\delta,t}=2\pi h^*\alpha_\delta+t[\omega_{\widetilde X}],
\]
where $\omega_{\widetilde X}$ is a fixed K\"ahler form. Our convention is
$c_1(E)=[\im\tr F_E/(2\pi)]$, and the geometric Ricci form is
$\operatorname{Ric}(\omega)=-\im\partial\bar\partial\log\det(\omega)$, whose class is $-2\pi c_1(K_{\widetilde X})$. Thus
$$2\pi c_1(K_{\widetilde X})+[\beta_{\delta,t}]=A_{\delta,t}+2\pi[D].$$
Choose defining sections $s_i$ of $D_i$ and smooth Hermitian metrics $h_i$
on $\mO_{\widetilde X}(D_i)$, and put $\Theta_i=\im F_{h_i}$, so that $[\Theta_i]=2\pi[D_i]$. Define
\[
 \theta_{i,\varepsilon}
 =\Theta_i+\sqrt{-1}\partial\bar\partial
   \log(|s_i|_{h_i}^2+\varepsilon^2)
 =P_{i,\varepsilon}+Q_{i,\varepsilon},
\]
where
\[
 P_{i,\varepsilon}
 =\frac{\varepsilon^2\sqrt{-1}
 \langle D's_i,D's_i\rangle_{h_i}}
 {(|s_i|_{h_i}^2+\varepsilon^2)^2}\geq0,
 \qquad
 Q_{i,\varepsilon}
 =\frac{\varepsilon^2}{|s_i|_{h_i}^2+\varepsilon^2}\Theta_i.
\]
Set $\Theta_\varepsilon=\sum_i a_i\theta_{i,\varepsilon}$. The klt
assumption gives $a_i>-1$, but does not imply $a_i\geq0$; thus
$\Theta_\varepsilon$ need not be semipositive.

In particular, $[\Theta_\varepsilon]=2\pi[D]$, and
$$-A_{\delta,t}+[\beta_{\delta,t}]-[\Theta_\varepsilon]
=-2\pi c_1(K_{\widetilde X}).$$
The twisted Aubin--Yau equation therefore has a K\"ahler solution
$\omega_\varepsilon\in A_{\delta,t}$ satisfying
\begin{equation}\label{twisted-Ricci}
 \operatorname{Ric}(\omega_\varepsilon)
 =-\omega_\varepsilon+\beta_{\delta,t}-\Theta_\varepsilon.
\end{equation}
We use the following error estimate.
\begin{claim}\label{error-vanishing}
For every $i$ and fixed $\delta,t>0$,
\[
 \lim_{\varepsilon\to0}\int_{\widetilde X}
 \frac{\varepsilon^2}{|s_i|_{h_i}^2+\varepsilon^2}
 \omega_{\widetilde X}\wedge\omega_\varepsilon^{n-1}=0.
\]
\end{claim}
\begin{proof}[Proof of the claim]
Choose a fixed K\"ahler form $\gamma_{\delta,t}\in A_{\delta,t}$ and write
$\omega_\varepsilon=\gamma_{\delta,t}+\im\partial\bar\partial\varphi_\varepsilon$.
Equation \eqref{twisted-Ricci} has the form
$$\omega_\varepsilon^n=e^{\varphi_\varepsilon}R_\varepsilon\,dV,\qquad
R_\varepsilon=e^{f_{\delta,t}}\prod_i(|s_i|_{h_i}^2+\varepsilon^2)^{a_i},$$
where $f_{\delta,t}$ and $dV$ are fixed and smooth. Since the divisor has simple normal crossings and $a_i>-1$, the functions $R_\varepsilon$ are uniformly bounded in $L^p$ for some $p>1$, and their integrals are bounded above and away from zero. Let
$$c_\varepsilon=\frac{A_{\delta,t}^n}{\int_{\widetilde X}R_\varepsilon\,dV}$$
and solve
$$(\gamma_{\delta,t}+\im\partial\bar\partial\psi_\varepsilon)^n
=c_\varepsilon R_\varepsilon\,dV,\qquad\sup\psi_\varepsilon=0.$$
Ko\l odziej's estimate gives a uniform bound for $\psi_\varepsilon$. Comparing the two determinants at a maximum and a minimum of $\varphi_\varepsilon-\psi_\varepsilon$ bounds $e^{\varphi_\varepsilon}$ above and below by $c_\varepsilon$, respectively, at those points. Hence $\varphi_\varepsilon$ is also uniformly bounded.

The Chern--Levine--Nirenberg capacity estimate for uniformly bounded potentials implies that
$\omega_{\widetilde X}\wedge\omega_\varepsilon^{n-1}$ has uniformly small mass in sufficiently small neighbourhoods of $D_i$. On the complement of such a neighbourhood, the factor
$\varepsilon^2/(|s_i|_{h_i}^2+\varepsilon^2)$ converges uniformly to zero. This proves the claim, as in \cite[Claim~9.5]{GGK19}.
\end{proof}

The metric $\omega_\varepsilon$ and the standard metric on the trivial
summand induce a metric $H_\varepsilon^{-1}$ on
$\mE_{\widetilde X}$. A direct computation with the natural Higgs field
gives
\begin{equation}\label{natural-mean-curvature}
 \sqrt{-1}\Lambda_{\omega_\varepsilon}
 F_{H_\varepsilon^{-1},\theta_{\widetilde X}}
 =\frac{n}{n+1}\Id_{\mE_{\widetilde X}}
 +\begin{pmatrix}
 (\omega_\varepsilon^{-1}
 (\Theta_\varepsilon-\beta_{\delta,t}))^\tau&0\\0&0
 \end{pmatrix}.
\end{equation}
Apply the Chern--Weil formula to $\mF_\delta$ and discard the nonpositive
term involving the second fundamental form. To estimate the discrepancy terms of either sign, put
\[
 D_{\rm abs}=\sum_i|a_i|D_i,\quad
 P_\varepsilon^{\rm abs}=\sum_i|a_i|P_{i,\varepsilon},\quad
 \rho_\varepsilon=\sum_i|a_i|
 \frac{\varepsilon^2}{|s_i|_{h_i}^2+\varepsilon^2}.
\]
Choose $C_0$ with $-C_0\omega_{\widetilde X}\leq\Theta_i
\leq C_0\omega_{\widetilde X}$. Then
\[
 -B_\varepsilon\leq\Theta_\varepsilon\leq B_\varepsilon,
 \qquad
 B_\varepsilon=P_\varepsilon^{\rm abs}
 +C_0\rho_\varepsilon\omega_{\widetilde X}.
\]
For an orthogonal projection $\pi$ and real forms
$-\zeta\leq\eta\leq\zeta$ with $\zeta\geq0$, Hermitian linear algebra
gives
\[
 \left|\tr\bigl(((\omega_\varepsilon^{-1}\eta)^\tau\oplus0)\pi\bigr)\right|
 \leq\tr_{\omega_\varepsilon}\zeta.
\]
Using this inequality in \eqref{natural-mean-curvature} gives
\begin{align}
 &\mu_{A_{\delta,t}}(\mF_\delta)
 -\mu_{A_{\delta,t}}(\mE_{\widetilde X}) \notag\\
 &\quad\leq C_1\left(
 [\beta_{\delta,t}]\cdot A_{\delta,t}^{n-1}
 +\int_{\widetilde X}P_\varepsilon^{\rm abs}
       \wedge\omega_\varepsilon^{n-1}
 +\int_{\widetilde X}\rho_\varepsilon\omega_{\widetilde X}
       \wedge\omega_\varepsilon^{n-1}\right).
 \label{positive-majorant-estimate}
\end{align}
Because
$\sum_i|a_i|\theta_{i,\varepsilon}$ represents $2\pi[D_{\rm abs}]$,
\[
 \int_{\widetilde X}P_\varepsilon^{\rm abs}
 \wedge\omega_\varepsilon^{n-1}
 \leq 2\pi D_{\rm abs}\cdot A_{\delta,t}^{n-1}
 +C_0\int_{\widetilde X}\rho_\varepsilon\omega_{\widetilde X}
 \wedge\omega_\varepsilon^{n-1}.
\]
Letting $\varepsilon\to0$ and applying Claim~\ref{error-vanishing}, we obtain
\begin{equation}\label{before-t-limit}
 \mu_{A_{\delta,t}}(\mF_\delta)
 -\mu_{A_{\delta,t}}(\mE_{\widetilde X})
 \leq C_2\bigl([\beta_{\delta,t}]\cdot A_{\delta,t}^{n-1}
 +D_{\rm abs}\cdot A_{\delta,t}^{n-1}\bigr).
\end{equation}
Now let $t\to0$. Every $D_i$ is $h$-exceptional, so the projection
formula gives $D_i\cdot(h^*\alpha_\delta)^{n-1}=0$. Since
$A_{\delta,t}\to2\pi h^*\alpha_\delta$ and
$$\mu_{2\pi h^*\alpha_\delta}(\mG)
=(2\pi)^{n-1}\mu_{h^*\alpha_\delta}(\mG),$$
we divide the limiting inequality by $(2\pi)^{n-1}$ and absorb the remaining fixed factor into $C$ to obtain \eqref{slope-estimate-main}. Thus we first take $\varepsilon\to0$, then $t\to0$, and finally $\delta\to0$.

Finally, any saturated invariant subsheaf of $\mE_X$ has slope at most
that of $\mF_{\delta,X}$. Using \eqref{natural-slope-gap} and letting
$\delta\to0$ in \eqref{slope-estimate-main} proves $K_X^{n-1}$-semistability.
\end{proof}

\subsubsection{Completion of the Miyaoka--Yau proof}
Put $\nu=\nu(K_X)$,
$V_\delta=\alpha_\delta^n$. Proposition
\ref{upperslope}, \eqref{natural-slope-gap} and \eqref{estimate} give
\begin{equation}\label{MY-asymptotic-BG}
 \widehat\Delta(\mE_X)\cdot\alpha_\delta^{n-2}
 \geq-C\frac{\delta^2\bigl([\omega_X]\cdot\alpha_\delta^{n-1}\bigr)^2}{V_\delta}.
\end{equation}
If $\nu=n$, the right-hand side is $O(\delta^2)$. If $\nu<n$, then
$V_\delta\asymp\delta^{n-\nu}$ and
$[\omega_X]\cdot\alpha_\delta^{n-1}=O(\delta^{n-1-\nu})$, so the right-hand side is
$O(\delta^{n-\nu})$.

Let $i=\min\{\nu,n-2\}$. For a nef class, the definition of numerical
dimension and positivity imply that $K_X^k=0$ numerically for $k>\nu$.
Consequently the binomial expansion is
\begin{align*}
 \widehat\Delta(\mE_X)\cdot\alpha_\delta^{n-2}
 &=\binom{n-2}{i}\delta^{n-2-i}
 \widehat\Delta(\mE_X)\cdot K_X^i\cdot[\omega_X]^{n-2-i}\\
 &\quad+O(\delta^{n-1-i}).
\end{align*}
Divide \eqref{MY-asymptotic-BG} by $\delta^{n-2-i}$ and let
$\delta\to0$. The error on the right is $O(\delta^2)$ if
$\nu\leq n-2$, $O(\delta)$ if $\nu=n-1$, and $O(\delta^2)$ if $\nu=n$.
Using \eqref{natural-discriminant}, we obtain
\[
 \left(2\widehat c_2(X)-\frac{n}{n+1}\widehat c_1(X)^2\right)
 \cdot K_X^i\cdot[\omega_X]^{n-2-i}\geq0,
\]
which proves Theorem~\ref{t1}.

\appendix

%-----------
\bigskip
	
\medskip

%{\bf Acknowledgements}: The authors would like to thank the editor and the referees for their help  and valuable comments.

%\vskip 1 true cm

\bibliographystyle{plain}

\end{document}